\documentclass[12pt,twoside]{amsart}
\usepackage{amsthm,amsfonts,amssymb,amscd,euscript}

\usepackage{comment}

\usepackage[matrix,arrow]{xy}

\author{Florin Ambro} 
\address{Institute of Mathematics ``Simion Stoilow'' of the Romanian Academy\\
P.O. BOX 1-764, RO-014700 Bucharest\\ 
Romania.}
\email{florin.ambro@imar.ro}
\author{Atsushi Ito}
\address{Graduate School of Mathematics, Nagoya University, Nagoya, 464-8602, Japan.}
\email{atsushi.ito@math.nagoya-u.ac.jp}

\newcommand{\isoto}{{\overset{\sim}{\rightarrow}}}

\newcommand{\Q}{{\mathbb Q}}
\newcommand{\Z}{{\mathbb Z}}
\newcommand{\N}{{\mathbb N}}
\newcommand{\R}{{\mathbb R}}

\newcommand{\bP}{{\mathbb P}} 
\newcommand{\bA}{{\mathbb A}} 

\newcommand{\cC}{{\mathcal C}}

\newcommand{\cI}{{\mathcal I}}

\newcommand{\cL}{{\mathcal L}}
\newcommand{\cO}{{\mathcal O}}

\newcommand{\bB}{{\mathbf B}}

\newcommand{\fm}{{\mathfrak m}}

\newcommand{\Bs}{\operatorname{Bs}}

\newcommand{\Char}{\operatorname{char}}
\newcommand{\codim}{\operatorname{codim}}
\newcommand{\Conv}{\operatorname{Conv}}

\newcommand{\emb}{\operatorname{emb}}

\newcommand{\Exc}{\operatorname{Exc}}

\newcommand{\Hom}{\operatorname{Hom}}

\newcommand{\im}{\operatorname{Im}}

\newcommand{\Int}{\operatorname{int}}

\newcommand{\length}{\operatorname{length}}

\newcommand{\mult}{\operatorname{mult}}

\newcommand{\ord}{\operatorname{ord}}

\newcommand{\Proj}{\operatorname{Proj}}

\newcommand{\Spec}{\operatorname{Spec}}
\newcommand{\Supp}{\operatorname{Supp}}

\newcommand{\vol}{\operatorname{vol}}

\newcommand{\width}{\operatorname{width}}
\newcommand{\rank}{\operatorname{rank}}

\theoremstyle{plain}
\newtheorem{thm}{Theorem}[section]

\newtheorem{lem}[thm]{Lemma}
\newtheorem{cor}[thm]{Corollary}

\newtheorem{prop}[thm]{Proposition}

\theoremstyle{definition}
\newtheorem{defn}[thm]{Definition}
\newtheorem{defnprop}[thm]{Definition-Proposition}

\newtheorem{exmp}[thm]{Example}

\newtheorem{rem}[thm]{Remark}

\newtheorem{ack}{Acknowledgments}   

\theoremstyle{remark}
\newtheorem{claim}[thm]{Claim}

\begin{document}

\bibliographystyle{amsalpha+}
\title{Successive minima of line bundles}
\maketitle

\begin{abstract} 
We introduce and study the successive minima of line bundles on proper algebraic varieties.
The first (resp.\ last) minima are the width (resp.\ Seshadri constant) of the line bundle 
at very general points. The volume of the line bundle is equivalent to 
the product of the successive minima. For line bundles on toric varieties, the
successive minima are equivalent to the (reciprocal of) usual successive minima
of the difference of the moment polytope.
\end{abstract} 






\section*{Introduction}


The motivation for this paper is the classical problem in Algebraic Geometry of finding numerical criteria for a linear system on an
algebraic variety to be non-empty, or to define a birational embedding. The equivalent problem in Geometry of Numbers
is finding numerical criteria for a convex body to contain sufficiently many lattice points.

For a complex projective manifold $X$, Demailly~\cite{De92} introduced the Seshadri constant $\epsilon(L,x)$ 
of an ample line bundle $L$ at a point $x$ of $X$. This invariant is constant if $x$ is very general, denoted 
$\epsilon(L)$ and called maximal Seshadri constant, and one can show that if $\epsilon(L)$ is sufficiently large, 
then the linear system $|K_X+L|$ is non-empty, and even defines a birational embedding (see~\cite{De92, EKL}, or Theorem~\ref{aj}).

For a convex body $\square\subset \R^d$, the flatness theorem of Khinchin states that if the lattice width $w$
of $\square$ is sufficiently large, then $\square$ contains a lattice point (see~\cite{KL88}). The starting point of 
this paper is to observe that the flatness theorem is just a special (toric) case of Demailly's generation for adjoint line bundles
in terms of maximal Seshadri constants. The key point is that on toric varieties, the maximal Seshadri constant 
is equivalent to the lattice width of the moment polytope:

\begin{thm}
Let $X=T_N\emb(\Delta)$ be a proper toric variety, of dimension $d$. Let $L$ be an invariant Cartier divisor on $X$,
with Seshadri constant $\epsilon$ at a general point of $X$. Let $\square_L\subset M_\R$ be the moment polytope, 
let $w$ be the lattice width of $\square_L$. Then 
$$
\frac{w}{d}\le \epsilon\le w.
$$
\end{thm}

 The equivalence between $w$ and $\epsilon$ not only reproves the flatness theorem, but generalizes
it as follows: 

\begin{thm} Let $\square\subset \R^d$ be a compact convex set, of lattice width $w$. 
If $w>d^2+d$, there exist $m_0,\ldots,m_d\in \Z^d \cap \Int(\square)$
such that $m_1-m_0,\ldots, m_d-m_0$ form a basis of $\R^d$.
\end{thm}

The proof of the equivalence involves the transference theorem of Mahler, and the fact that $w$ is 
the first minimum of Minkowski for the polar body $(\square_L-\square_L)^*\subset N_\R$. Since $\epsilon$
is proportional to the first minimum of Minkowski, we may think of $1/\epsilon$ as a $d$-th
successive minimum of $\square_L-\square_L\subset M_\R$ via the transference theorem. The question arises if the other successive minima of this $0$-symmetric body
have a meaning on algebraic varieties which are not necessarily toric, and if they have a connection with Seshadri constants.
The answer turns out to be positive! We introduce in this paper a sequence of successive minima for a line bundle on an algebraic 
variety, such that the last minimum coincides with the classical Seshadri constant, and in the toric case, these minima are
equivalent to the (reciprocal of) successive minima of the $0$-symmetric body associated to the moment polytope.

Let $X$ be a proper complex variety, let $L$ be a Cartier divisor on $X$.
Let $x\in X$ be a point. For a real number $t\ge 0$, we define $\Bs|\cI_x^{t+}L|_\Q$ to be the common zero
locus of sections $s\in \Gamma(nL)\ (n\ge 1)$ such that $s$ vanishes at $x$ of order strictly larger than $nt$.
Then $\Bs|\cI_x^{t+}L|_\Q$ is a closed subset of $X$, increasing with respect to $t$, and equal to $X$ for $t$ sufficiently large.
For integers $i\ge 1$, the {\em $i$-th successive minimum of $L$ at $x$} is defined as 
$$
\epsilon_i(L,x)=\inf\{t\ge 0; \codim_x \Bs|\cI_x^{t+} L|_\Q < i  \}.
$$
We obtain a sequence $\epsilon_1(L,x)\ge \epsilon_2(L,x)\ge \cdots \ge \epsilon_d(L,x)\ge 0=\epsilon_{d+1}(L,x)$,
where $d=\dim X$. 
As we will see,
the invariant $\epsilon_i(L,x)$ does not depend on a very general point $x$, is denoted $\epsilon_i(L)$, 
and called the {\em $i$-th minimum of $L$ at a very general point}. If $\kappa(L)\le 0$, then $\epsilon_1(L)=0$. If 
$L$ has Iitaka dimension $\kappa\ge 1$, it turns outs that $\epsilon_\kappa(L)>0=\epsilon_{\kappa+1}(L)$.
We can show that $\epsilon_i(L)$ are numerical invariants in case $L$ is big.

If $L$ is semiample and $x\in X$ is a smooth point, $\epsilon_d(L,x)$ coincides with the Seshadri constant
$\epsilon(L,x)$ introduced by Demailly~\cite{De92}. In particular, if $L$ is semiample, $\epsilon_d(L)$ coincides with the 
maximal Seshadri constant of $L$, usually denoted $\epsilon(L)$.

The invariant $\epsilon_1(L,x)$ is the largest asymptotic multiplicity that can be imposed at $x$ 
on sections of multiplies of $L$. It appears in the work of Nakamaye~\cite{Nak03b}. If $x \in X$ is a smooth point, it coincides with the
width of $L$ at the geometric valuation induced by the exceptional divisor on the blow-up at $x$ (see~\cite{Amb16}).
Due to this property, we call $\epsilon_1(L,x)$ the {\em width of $L$ at $x$}.
We also relate the volume of $L$ with the successive minima. By a classical argument for counting jets, 
we have inequalities
$$
\epsilon_d(L,x)\le \sqrt[d]{\frac{\vol(L)}{\mult_x(X)}}\le \epsilon_1(L,x).
$$
The second main result of the paper is the analogue of Minkowski's second main theorem,
namely the volume of $L$ is equivalent to the product of successive minima of $L$ at very general points:

\begin{thm}\label{m2m}
Let $X$ be a proper complex variety of dimension $d$, let $L$ be a Cartier divisor on $X$. Let $\vol(L)$ be the volume
of $L$, let $\epsilon_i(L)$ be the successive minima of $L$ at very general points. Then 
$$
\prod_{i=1}^d\epsilon_i(L)\le \vol(L)\le d!\cdot \prod_{i=1}^d\epsilon_i(L).
$$
\end{thm}

In particular, for a big divisor $L$ on $X$ we have inequalities
$$
1\le \frac{\vol(L)}{\epsilon_d(L)^d}\le d!\cdot  (\frac{\epsilon_1(L)}{\epsilon_d(L)})^d.
$$

Understanding when the ratio $\epsilon_1(L)/\epsilon_d(L)$ is too large is an interesting problem (see Nakamaye~\cite{Nak03b}).
Finally, we show that for line bundles on toric varieties, our successive minima are equivalent to the
reciprocal of usual ones:

\begin{thm}
Let $X=T_N\emb(\Delta)$ be a proper toric variety, of dimension $d$. Let $L$ be a big invariant Cartier divisor on $X$,
with successive minima $\epsilon_i$ at a very general point. Let $\square_L\subset M_\R$ be the moment polytope, 
let $\lambda_1,\ldots,\lambda_d$ be the successive minima of the $0$-symmetric convex body $\square_L-\square_L\subset M_\R$,
let $\lambda_1^*,\ldots,\lambda_d^*$ be the successive minima of the polar body $(\square_L-\square_L)^*\subset N_\R$.
Then $\epsilon_i,\lambda_i^{-1},\lambda^*_{d-i+1}$ are all equivalent. More precisely,
$$
1\le \epsilon_i\cdot \lambda_i \le d\cdot \frac{\epsilon_i}{\lambda_{d-i+1}^*}\le d(d-i+1).
$$
\end{thm}

Our definition of successive minima is inspired by the equivalent definition of Seshadri constants due to Eckl~\cite[Theorem 1.1]{Eck},
and by the methods developed by Ein, K\"uchle, Lazarsfeld~\cite{EKL} and Nakamaye~\cite{Nak03,Nak03b,Nak05}, 
while studying effective lower bounds for maximal Seshadri constants. Theorem~\ref{m2m}, in the toric case, is 
just a restatement of Minkowski's second main theorem. It is implicit in the work of 
Nakamaye when $X$ is a surface~\cite[Proof of Corollary 3]{Nak03b}, 
and our proof is similar to his, except that we introduce certain convex polytopes in 
his method of counting of jets. Eventually, Theorem~\ref{m2m} reduces to a simple postulation problem in the projective 
space (Proposition~\ref{post}).

A technical improvement in the study of Seshadri constants is that we no longer require positivity of the line bundles, or
that the ambient space is normal.
On the other hand, our invariants may not be numerical if the line bundle is not big. We hope that successive minima,
and other tools from the Geometry of Numbers, may be useful in the study of adjoint linear systems.

We outline the content of this paper. In Section 1 we setup the notation, and recall basic facts about multiplicities,
and linear systems. We also prove an inequality for symbolic powers of ideals in the projective space, 
which is elementary, but new to us. In Section 2, we define successive minima for a line
bundle, study some basic properties, and compare it with Seshadri constants. In Section 3 we prove that the volume
is equivalent to the product of the successive minima. In Section 4, we estimate the successive minima for line
bundles on toric varieties. In Section 5, we generalize the flatness theorem of Khinchin, as an application of known
results on adjoint linear systems, plus the estimates in Section 4. In Section 6 we recall the statements of the second
main theorem of Minkowski, and the transference theorem of Mahler. We also present a proof of the transference theorem
in dimension two, with sharp bounds.

\begin{ack} The first author is grateful to Max Planck Institut f\"{u}r Mathematik in Bonn for its hospitality and financial support.
The second author was supported by Grant-in-Aid for Scientific Research 17K14162. We would like to thank the referee 
for useful comments and suggestions.
\end{ack}


\section{Preliminary}


Throughout this paper, an {\em algebraic variety} $X$ is a scheme of finite type,
reduced and irreducible, defined over an algebraically closed field $k$, of
characteristic zero (the assumption $\Char k=0$ is used in Lemma~\ref{ve} and
its consequences in Section 3, and in Theorem~\ref{aj}).


\subsection{Vanishing orders}

Let $X$ be an algebraic variety, $\cL$ an invertible $\cO_X$-module, and $s\in \Gamma(X,\cL)$ 
a global section. The {\em order of $s$ at a closed point} $x\in X$, denoted $\ord_x(s)$, is the supremum after
all integers $n\ge 0$ such that $s_x\in \cI_x^{n}\cdot \cL_x$. The order is 
$+\infty$ if $s=0$, and a non-negative integer if $s\ne 0$.

If $f\colon X'\to X$ is a morphism of algebraic varieties and $f(x')=x$,
then $f^*s\in \Gamma(X',f^*\cL)$ and $\ord_{x'}(f^*s)\ge \ord_x(s)$.
In particular, if $x\in X' \subseteq X$ is a closed subset, then $\ord_x(s)\le \ord_x(s|_{X'})$.

Suppose $X$ is smooth at the general point of a subvariety $Z\subseteq X$. The order of $s$ at $Z$
is defined as the order of $s$ at a general point of $Z$.


\subsection{Order bounds}

For an effective Cartier divisor $D$ on $X$, the order of $D$ at $x$, denoted $\ord_x(D)$, is defined
as the order at $x$ of a local equation of $D$.

\begin{lem}\label{bm}
Let $X$ be a projective algebraic variety, of dimension $d$,
let $A$ be a very ample divisor, and $D$ an effective 
Cartier divisor on $X$. Then $\ord_x(D)\le (D\cdot A^{d-1})$ for every closed point $x\in X$.
\end{lem}

\begin{proof}
Since $A$ is very ample, there exist $H_1,\ldots,H_{d-1}\in |A|$ such that 
$C=\cap_{i=1}^{d-1}H_i$ is an effective cycle passing through $x$, none of its components being contained in $D$. 
Then 
$$
(D\cdot A^{d-1})=(D\cdot C)\ge \ord_x(D|_C) \cdot \mult_x(C) \ge \ord_x(D) \cdot \mult_x(C)\ge \ord_x(D).
$$
We note that the first inequality follows from~\cite[page 233]{Fulton} even if $X$ is singular.
\end{proof}

\subsection{Volume}


Let $X$ be a proper algebraic variety, of dimension $d$. Let $L$ be a Cartier divisor on $X$,
let $R(L)=\oplus_{n\ge 0}\Gamma(X,\cO_X(nL))$ be the induced $\N$-graded ring. Let 
$R=\oplus_{n\ge 0}R_n\subseteq R(L)$ be a graded subalgebra. The {\em volume of $R$}
is defined as 
$$
\vol(R)=\limsup_{n\to \infty}\frac{\dim_kR_n}{n^d/d!}.
$$
The $\limsup$ is in fact a limit over sufficiently divisible $n$.
The {\em volume of $L$}, denoted $\vol(L)$, is defined as the volume of $R(L)$.

We usually denote $\Gamma(X,\cO_X(nL))$ by $\Gamma(nL)$.


\subsection{Iitaka map}


One may define the Iitaka dimension and Iitaka map for divisors on a variety which may not
be normal. These are birational invariants if the ambient space is normal. 

Let $X$ be a proper variety, let $L$ be a Cartier divisor on $X$. For $n\ge 1$ such that $\Gamma(nL)\ne 0$, 
let $|nL|$ be the induced linear system on $X$. A basis $s_0,\ldots,s_N$ of $\Gamma(nL)$ induces
a rational map $\phi_{|nL|}\colon X\dashrightarrow |nL|$, with image $X_n$. 
Define the {\em Iitaka dimension} $\kappa(L)$ of $L$ to be the maximum of $\dim X_n$ after all such $n$, 
and $-\infty$ if $\Gamma(nL)=0$ for all $n\ge 1$.
 
Suppose $\kappa(L)\ge 0$. For each $n\ge 1$ such that $\Gamma(nL)\ne 0$, let $Q_n\subseteq k(X)$ 
be the field over $k$ generated by $\{\frac{s'}{s}; s',s\in \Gamma(nL)\setminus 0\}$. The dominant rational map
$\phi_{|nL|}\colon X\dashrightarrow X_n$ induces an isomorphism $k(X_n)\isoto Q_n$. The union 
$\cup_{|nL|\ne \emptyset} Q_n$ is a subfield $Q$ of $k(X)$, since $Q_n\cup Q_{n'}\subseteq Q_{n+n'}$ if 
$|nL|$ and $|n'L|$ are non-empty. Since $Q$ is finitely generated over $k$, there exists an integer $m$ such 
that $Q_m=Q$. For every $n\ge 1$, we have $Q_m=Q_{nm}$, that is the rational maps 
$\phi_{|mL|}\colon X\dashrightarrow X_m$ and 
$\phi_{|nmL|}\colon X\dashrightarrow X_{nm}$ differ by a {\em birational} map $X_{nm}\dashrightarrow X_m$. 
We call $\phi_{|mL|}\colon X\dashrightarrow X_m$ the {\em Iitaka map} of $L$. The dimension of $X_m$
is $\kappa(L)$.

- Suppose $\kappa(L)\ge 0$ and $X$ is normal. The normalization of the graph of the Iitaka map
$\phi_{|mL|}\colon X\dashrightarrow X_m$ induces a diagram 
\[ \xymatrix{
 & X' \ar[dl]_\mu \ar[dr]^f  & \\
    X    \ar@{.>}[rr]     &       &   X_m
} \]
such that $X'$ is proper and normal, $\cO_X=\mu_*\cO_{X'}$, and 
$\kappa(\mu^*L|_{X'_y})=0$ for a very general point $y\in X_m$. Note that $X'_y$ is connected.

- Suppose $L$ is {\em semiample}, that is the linear system $|nL|$ has no base points for some 
$n\ge 1$. Then the Iitaka map $\phi\colon X\to Y$ is regular, $\cO_Y=\phi_*\cO_X$, and 
$nL\sim \phi^*A$ for some integer $n \ge 1$ and a very ample divisor $A$ on $Y$.

- Suppose $\mu\colon X'\to X$ is a proper birational morphism. If $X$ and $X'$ are normal, it follows
that $\cO_X=\mu_*\cO_{X'}$. In particular, $\Gamma(nL)=\Gamma(n\mu^*L)$ for every $n\ge 1$.
Therefore the rational maps induced by non-empty linear systems $|nL|$ and $|n\mu^*L|$ differ by 
$\mu$. It follows that $\kappa(L)=\kappa(\mu^*L)$.

A morphism of algebraic varieties $f\colon X\to Y$ is called a {\em contraction} if 
$\cO_Y=f_*\cO_X$.


\subsection{Isolated base points of linear systems}

Let $X$ be a proper algebraic variety, let $L$ be a Cartier divisor on $X$ such that $\Gamma(L)\ne 0$. 
A basis $s_0,\ldots,s_N$ of $\Gamma(L)$
induces a rational map $\phi\colon X\dashrightarrow \bP^N=|L|$. Let $U=X\setminus \Bs|L|$,
so that $\phi$ is regular on $U$. Then $\Bs|\cI_x(L)|\cap U=\phi^{-1}\phi(x)\cap U$ for every $x\in U$.

\begin{lem}\label{icr}
Let $f\colon X\to Y$ be a proper contraction of algebraic varieties, let $x\in X$ be a closed point such that
$\dim_x X_{f(x)} =0$, where $X_{f(x)}= f^{-1}f(x)$. Then $f$ is an isomorphism over a neighborhood of $f(x)$.
\end{lem}

\begin{proof} By the semicontinuity of the local dimension of the fibers of $f$,
there exists an open subset $U'\ni x$ such that $\dim_{x'}X_{f(x')}=0$ for all $x'\in U'$. Let 
$V=Y\setminus f(X\setminus U')$. Since $X_{f(x)}$ is $0$-dimensional and connected, it coincides 
with $\{x\}$. Therefore $f(x)\in V$. Let $U=f^{-1}(V)$. Then $x\in U\subseteq U'$. Therefore 
$f|_U\colon U\to V$ is a proper contraction with finite fibers, hence a finite contraction, hence an isomorphism.
\end{proof}

\begin{lem}[Lemma 5.2.24~\cite{Lazbook}]\label{bl}
Let $X$ be an algebraic variety. Let $\cI\subseteq \cO_X$ be an ideal sheaf,
let $f\colon Y\to X$ be the blow-up of $X$ along $\cI$. Let $E$ be the exceptional
divisor on $Y$, defined by $\cI\cdot \cO_Y=\cO_Y(-E)$. Then the natural homomorphism
$
\cI^n\to f_*\cO_Y(-nE)
$
is an isomorphism for every $n\ge c(\cI)$, where $c(\cI)\ge 0$ is a constant depending only on $\cI$.
Moreover, $c(\cI)=0$ if $\cI$ is the ideal of a smooth point of $X$.
\end{lem}

\begin{proof}
The statement is local, so we may suppose $X=\Spec R $ and $Y=\Proj S$, where $S$
is the $R$-graded ring $\oplus_{n\ge 0}I^n$, for some ideal $I\subseteq R$. Then 
$\cO_Y(-nE)=\cO_Y(n)=\tilde{S(n)}$. The natural homomorphism
$
S\to \oplus_{n\ge 0}\Gamma(Y,\tilde{S(n)})
$
is an isomorphism in degrees $n\ge c(I)$. This gives the claim.

If $\cI$ is the ideal of a smooth point of $X$, then $\cI$ is generated by a regular sequence.
One checks then that $c(\cI)=0$.
\end{proof}

\begin{lem}[Proposition 1.1~\cite{ELMNP}]\label{ibp}
Let $X$ be a proper algebraic variety, let $L$ be a Cartier divisor on $X$ and $x\in X$ a closed point
such that $\Bs |L|\cap U\subseteq \{x\}$ for some open neighborhood $U$ of $x$ in $X$.
Then $x\notin \Bs|nL|$ for $n\gg 0$.
\end{lem}

\begin{proof}
Let $Z$ be the base locus {\em subscheme} of $|L|$ outside $x$ (possibly empty). 
Let $f\colon Y\to X$ be the blow-up of $X$ along $\cI_Z$. Let $\cI_Z\cdot \cO_Y=\cO_Y(-E)$, so 
that $E$ is the exceptional divisor. 
Since $f$ is an isomorphism over a neighborhood of $x$, we may identify $x$ with a point of $Y$.

A basis of $\Gamma(X,\cI_Z(L))=\Gamma(X,\cO_X(L))$ induces sections of $\Gamma(Y,\cO_Y(f^*L-E))$ 
whose common zero locus is away from $E$. Therefore $\Bs|f^*L-E|\subseteq \{x\}$. By~\cite[Theorem 1.10]{Fuj83}, 
$x\notin \Bs|n(f^*L-E)|$ for $n\ge n_0$. By Lemma~\ref{bl}, we may identify
$\Gamma(X,\cI_Z^n(nL))=\Gamma(Y,\cO_Y(nf^*L-nE))$ for every $n\ge c(\cI_Z)$. 
Therefore $x\notin \Bs|\cI^n_Z(nL)|$ for $n\ge \max(n_0,c(\cI_Z))$. 
Since $x\notin Z$, we obtain $x\notin \Bs|nL|$ for $n\gg 0$.
\end{proof}

\begin{lem}[Ample reduction]\label{rda}
Let $X$ be proper, $L$ a Cartier divisor and $x\in X$ a closed point such that $\Bs|\cI_x(L)|=\{x\}$ near $x$.
Let $Z$ be the base locus {\em subscheme} of $|L|$, away from $x$ (possibly empty). Let $f \colon X'\to X$
be the blow-up of $X$ along $\cI_Z$, with exceptional divisor $E$. Then there exists a proper 
contraction $g\colon X' \to Y$ to a projective algebraic variety $Y$, such that the following properties hold:
\begin{itemize}
\item[a)] $f^*L-E\sim g^*A$ for some ample Cartier divisor $A$ on $Y$, and 
\item[b)] $f$ is an isomorphism over a neighborhood of $x$, and $g$ is an isomorphism over
a neighborhood of $g(x)$.
\end{itemize}
\[ \xymatrix{
 & X' \ar[dl]_f \ar[dr]^g  & \\
    X    \ar@{.>}[rr]     &       &   Y
} \]
We obtain natural homomorphisms
$
R(L)\supseteq \oplus_{n\ge 0}\Gamma(\cI_Z^n(nL))\stackrel{\alpha}{\to} R(f^*L-E) \simeq R(A),
$
and $\alpha$ is an isomorphism in sufficiently large degrees.
\end{lem}

\begin{proof} By construction, $|f^*L-E|$ is base point free near $E$. Therefore 
$\Bs|f^*L-E|\subseteq \{x\}$. By the proof of Lemma~\ref{ibp}, $|n(f^*L-E)|$ is base point free
for $n\gg 0$. The Iitaka map induces the contraction $g$, with property a).
For b), it is clear that $f$ is an isomorphism over a neighborhood of $x$. We identify
$x$ with a point of $X'$. Since $\Bs|\cI_x(L)|=\{x\}$ near $x$, we obtain 
$\Bs|\cI_x(f^*L-E)|=\{x\}$ near $x$. Therefore $\Bs|\cI_x(nf^*L-nE)|=\{x\}$ near $x$,
for all $n\ge 1$. Therefore $g^{-1}g(x)=\{x\}$ near $x$. By Lemma~\ref{icr}, $g$ is an
isomorphism over a neighborhood of $g(x)$. Finally, $\alpha_n$ is an isomorphism 
for $n\ge c(\cI_Z)$.
\end{proof}


\subsection{Generation of jets}

Let $X$ be an algebraic variety, $L$ a Cartier divisor on $X$, $x\in X$ a closed point, and $p\ge 0$ an integer.
We say that {\em $L$ generates $p$-jets at $x$} if $\Gamma(L)\to \cO_x/\cI_x^{p+1}$ is surjective.
By induction and the snake lemma, this is equivalent to the following property: 
the $\cO_X$-module $\cI_x^n(L)$ is generated by global sections at $x$, for every integer $0\le n\le p$.

\begin{lem}\label{lj}
Let $f\colon Y\to X$ be the blow-up at $x$, with exceptional divisor $E$. Suppose 
$\cI_x^{p+1}\isoto f_*(\cI_E^{p+1})$ and $R^1f_*(\cI_E^{p+1})=0$. If 
$\Gamma(f^*L)\to \Gamma(\cO_{(p+1)E})$ is surjective, then $L$ generates $p$-jets at $x$.
The converse holds if the natural homomorphism $H^1(L)\to H^1(f_*\cO_Y(L))$ is injective.
\end{lem}

\begin{proof} Consider the commutative diagram
\[ \xymatrix{
 0 \ar[r] & f_*(\cI_E^{p+1}) \ar[r] & f_*\cO_Y \ar[r]  & \Gamma(\cO_{(p+1)E}) \ar[r] & 0 \\
 0 \ar[r] & \cI_x^{p+1} \ar[r] \ar[u] & \cO_X \ar[r] \ar[u] & \cO_X/\cI_x^{p+1} \ar[r]\ar[u] & 0
} \]
The bottom row is exact. The top is also exact, since $R^1f_*(\cI_E^{p+1})=0$. The first vertical arrow
is an isomorphism, by assumption. The second vertical arrow is injective. Therefore the third
arrow is injective as well. We deduce that the second and third vertical arrows have isomorphic 
cokernels, denoted by $\cC$. Ignoring the first vertical arrow, tensoring with $\cO_X(L)$ 
and passing to global sections on $X$, we obtain a commutative diagram 
\[ \xymatrix{
 0 \ar[r] & \Gamma(L) \ar[r] \ar[d]^a & \Gamma(f^*L) \ar[r]^c  \ar[d]^b   & \Gamma(\cC) \ar[d]^= \\
 0 \ar[r] & \cO_X/\cI_x^{p+1} \ar[r]  & \Gamma(\cO_{(p+1)E})  \ar[r] & \Gamma(\cC) 
 } \]
with exact rows. If $b$ is surjective, then $a$ is surjective. The converse holds if $c$ is surjective.
From the exact sequence 
$$
0\to \Gamma(L)\to \Gamma(f_*\cO_Y(L))\stackrel{c}{\to} \Gamma(\cC)\to H^1(L)\to H^1(f_*\cO_Y(L)),
$$
we see that $c$ is surjective if and only if $H^1(L)\to H^1(f_*\cO_Y(L))$ is injective.
\end{proof}

If $x$ is a smooth point, then $\cI_x^{p+1}\isoto f_*(\cI_E^{p+1})$ and $R^1f_*(\cI_x^{p+1})=0$ for all $p\ge 0$,
and $\cO_X=f_*\cO_Y$. Therefore $L$ generates $p$-jets at $x$ if and only if 
$\Gamma(f^*L)\to \Gamma(\cO_{(p+1)E})$ is surjective.


\subsection{Postulation}


For an integer $p\ge 0$, denote $\N^d(p)=\{(\alpha_1,\ldots,\alpha_d)\in \N^d; \sum_{i=1}^d \alpha_i=p\}$. 
For a finite set $A$, denote by $|A|$ its cardinality. 
We will need the following property on the symbolic powers of ideals in the projective space.

\begin{prop}\label{post} 
Let $Z_i\subset \bP^d$ be an irreducible subvariety of codimension $i$, for $1\le i \le d$.
Let $p_1,\ldots,p_d,q\ge 0$ be integers. Then 
$$
h^0(\bP^d,\cap_{i=1}^d \cI_{Z_i}^{(p_i)}\otimes \cO_{\bP^d}(q))\le 
|\cap_{i=1}^d \{\alpha\in \N^{d+1}(q); \alpha_i+\cdots+\alpha_d \le q-p_i \} |.
$$ 
When $Z_1\supset Z_2\supset \cdots \supset Z_d$ is a linear flag in $\bP^d$, the equality holds.
\end{prop}

\begin{proof}
1) Suppose $Z_1\supset Z_2\supset \cdots \supset Z_d$ is a linear flag in $\bP^d$.
Choose coordinates $[z_0:\cdots:z_d]$ on $\bP^d$ such that $I_{Z_i}=(z_0,\ldots,z_{i-1})$ for all $i$. 
Then $\Gamma(\bP^d,\cap_{i=1}^d \cI_{Z_i}^{(p_i)}\otimes \cO_{\bP^d}(q))$ is the $k$-vector
space with monomial basis $z_0^{\alpha_0}\cdots z_d^{\alpha_d}$, where 
$\alpha\in \N^{d+1}(q)$ such that $\alpha_0+\cdots+\alpha_{i-1}\ge p_i$ for all $i$.
Equivalently, $\alpha_i+\cdots+\alpha_d \le q-p_i$ for all $i$.
Therefore the inequality becomes an equality if $Z_\bullet$ is a linear flag.

2) If $d=1$, the equality holds from 1). Suppose $d>1$, and assume by induction that the inequality
holds in smaller dimension. Let $H\colon \{\lambda=0\}$ be a general hyperplane with respect to 
$Z_\bullet$.  In particular, $H$ does not contain the point $Z_d$. The short exact sequence 
$$
0\to \Gamma(\bP^d,\cO_{\bP^d}(q-1))\stackrel{\otimes \lambda}{\to} \Gamma(\bP^d,\cO_{\bP^d}(q))\to 
\Gamma(H,\cO_H(q))\to 0
$$
induces an exact sequence 
$$
0\to \Gamma(\bP^d,\cap_{i=1}^d \cI_{Z_i}^{(p_i)}\otimes\cO_{\bP^d}(q-1))\stackrel{\otimes \lambda}{\to} 
\Gamma(\bP^d,\cap_{i=1}^d \cI_{Z_i}^{(p_i)}\otimes\cO_{\bP^d}(q))\to 
\Gamma(H,\cap_{i=1}^{d-1}\cI_{Z_i\cap H}^{(p_i)}\otimes\cO_H(q)).
$$

Indeed, $\Gamma(\bP^d,\cap_{i=1}^d \cI_{Z_i}^{(p_i)}\otimes\cO_{\bP^d}(q))$ consists of homogenous 
polynomials $F$ of degree $q$ such that $\ord_x(F)\ge p_i$ for every $x\in Z_i$, for all $i$. 
Then $\ord_x(F|_H)\ge p_i$ for every $x\in Z_i\cap H$. If $F|_H=0$, then $F=\lambda F'$ for some
homogeneous polynomial $F'$ of degree $q-1$. For $x\in Z_i\setminus H$ we have $\ord_x(F)=\ord_x(F')$. 
Then $F'$ vanishes to order at least $p_i$ at a general point of $Z_i$, hence everywhere on $Z_i$.

Moreover, the above exact sequence extends to a short exact sequence if $Z_\bullet$ is a linear flag:
$$
0\to \Gamma(\bP^d,\cap_{i=1}^d \cI_{Z_i}^{(p_i)}\otimes\cO_{\bP^d}(q-1))\stackrel{\otimes \lambda}{\to} 
\Gamma(\bP^d,\cap_{i=1}^d \cI_{Z_i}^{(p_i)}\otimes\cO_{\bP^d}(q))\to 
\Gamma(H,\cap_{i=1}^{d-1}\cI_{Z_i\cap H}^{(p_i)}\otimes\cO_H(q))\to 0.
$$
Indeed, we may choose coordinates as in 1), and suppose $\lambda=z_d$. The
part $\alpha_d\ge 1$ of the set
$$
\cap_{i=1}^d \{\alpha\in \N^{d+1}(q); \alpha_i+\cdots+\alpha_d \le q-p_i \}
$$
corresponds to a monomial basis of 
$\Gamma(\bP^d,\cap_{i=1}^d \cI_{Z_i}^{(p_i)}\otimes\cO_{\bP^d}(q-1))$, and the part $\alpha_d=0$
corresponds to a monomial basis of $\Gamma(H,\cap_{i=1}^{d-1}\cI_{Z_i\cap H}^{(p_i)}\otimes\cO_H(q))$.
Since the dimensions add up, the sequence is exact to the right as well.

Denote $f(q)=h^0(\bP^d,\cap_{i=1}^d \cI_{Z_i}^{(p_i)}\otimes \cO_{\bP^d}(q))$. 
For each $1\le i\le d-1$, choose an irreducible subvariety $W_i \subseteq Z_i\cap H$, of codimension $i$ in $H\simeq \bP^{d-1}$.
The exact sequence gives
$$
f(q)-f(q-1)\le h^0(\bP^{d-1},\cap_{i=1}^{d-1}\cI_{W_i}^{(p_i)}(q)).
$$
Denote $g(q)=h^0(\bP^d,\cap_{i=1}^{d} \cI_{L_i}^{(p_i)}\otimes \cO_{\bP^d}(q))$, where $L_\bullet$ is a 
linear flag in $\bP^d$. The exact sequence gives 
$$
g(q)-g(q-1)= h^0(\bP^{d-1},\cap_{i=1}^{d-1}\cI_{L_i \cap H}^{(p_i)}(q)).
$$
The inequality in dimension $d-1$ gives 
$$
h^0(\bP^{d-1},\cap_{i=1}^{d-1}\cI_{W_i}^{(p_i)}(q))
\le 
h^0(\bP^{d-1},\cap_{i=1}^{d-1}\cI_{L_i \cap H}^{(p_i)}(q)).
$$
We deduce 
$$
f(q)-f(q-1)\le g(q)-g(q-1) \ \forall q.
$$
We have $\Gamma(\bP^d,\cI_{Z_d}^{(p_d)}(q))=0$ for $q<p_d$. Therefore $f(q)=g(q)=0$ for $q<p_d$.
By increasing induction on $q\ge p_d$, the above inequality gives $f(q)\le g(q)$ for all $q$.
\end{proof}

For real numbers $t_1, \ldots, t_d$, define a compact convex set in $\R^d$ by
$$
\square(t_1,\ldots,t_d)=\cap_{i=1}^d\{x\in \R^d_{\ge 0}; x_i+\cdots+x_d \le t_i \}.
$$
Note that $\square(t_1,\ldots,t_d)$ is not empty if and only if $t_1,\ldots,t_d\ge 0$.
If $t_1\ge \cdots \ge t_d\ge 0$, we may compute 
$\vol \square(t_1)=t_1$, $2\vol\square(t_1,t_2)=2t_1t_2-t_2^2$, and 
$6\vol \square(t_1,t_2,t_3)=6t_1t_2t_3-3t_3^2t_1-3t_3t_2^2+t_3^3$.

\begin{lem}\label{vol} Suppose $t_i\ge 0$ for all $i$. Then $\vol \square(t_1,\ldots,t_d)\le \prod_{i=1}^d t_i$.
\end{lem}

\begin{proof}
We have an inclusion $\square(t_1,\ldots,t_d)\subset [0,t_1]\times \square(t_2,\ldots,t_d)$. Therefore
$$
\vol\square(t_1,\ldots,t_d)\le t_1\cdot \vol\square(t_2,\ldots,t_d).
$$ 
By induction on $d$, the desired inequality holds.
\end{proof}

Using this convex set, we may restate Proposition~\ref{post} as follows:
$$
h^0(\bP^d,\cap_{i=1}^d \cI_{Z_i}^{(p_i)}\otimes \cO_{\bP^d}(q))\le |\Z^d\cap \square(q-p_1,\ldots,q-p_d)|.
$$


\section{Successive minima of line bundles}


Let $X$ be a proper algebraic variety, of dimension $d$. Let $L$ be a Cartier divisor,
with induced graded ring $R=\oplus_{n\ge 0}\Gamma(nL)$. Let $x\in X$ be a closed point. 
For a real number $t\ge 0$, denote 
$$
\Bs|\cI_x^{t+} L|_\Q=\cap_{n\ge 1}\{Z(s); s\in R_n,\ord_x(s)>nt \}.
$$ 
It is a closed subset of $X$. If $t\le t'$, then $\Bs|\cI_x^{t+} L|_\Q\subseteq \Bs|\cI_x^{t'+} L|_\Q$.
One can rewrite $\Bs|\cI_x^{t+} L|_\Q$ as the intersection of the base locus $\Bs|\cI_x^p(qL)|$, 
after all integers $p,q\ge 1$ such that $p>qt$. Since 
$$
\oplus_{n\ge 0}\{s\in R_n;\ord_x(s)>nt \}
$$
is a graded ring and $X$ is Noetherian, there exists $r\ge 1$ such that 
$\Bs|\cI_x^{t+} L|_\Q=\Bs|\cI_x^{\lfloor rt\rfloor+1}(rL)|$.

\begin{exmp}
Suppose $L$ is semiample. Let $\phi\colon X\to Y$ be the Iitaka map
of $L$. Then $\Bs|\cI_x^{0+} L|_\Q=\phi^{-1}\phi(x)$. It is connected,
since $\phi$ is a contraction.

We have $rL\sim f^*A$ for some $r\ge 1$ and $A$ ample on $Y$. 
Then for every $t\ge 0$, we have $\Bs|\cI_x^{t+} L|_\Q\subseteq f^{-1}(\Bs|\cI_{f(x)}^{rt+} A|_\Q)$,
and equality holds if $f$ is smooth at $x$.
\end{exmp}

\begin{lem}\label{ug}
Given $u\ge 0$, there exists $\epsilon(u)>0$ such that $\Bs|\cI_x^{t+} L|_\Q$ is constant 
for $t\in [u,u+\epsilon(u))$.
\end{lem}

\begin{proof} Denote $B_t=\Bs|\cI_x^{t+} L|_\Q$.
We have $B_u=\cap_{l\ge 1}B_{u+\frac{1}{l}}$. By monotonicity, and since $X$ is Noetherian,
there exists $l\ge 1$ such that $B_u=B_{u+\frac{1}{l}}$.
\end{proof}

\begin{lem} There exists a constant $c$, depending only on $X$ and $L$, such that 
$\Bs|\cI_x^{t+} L|_\Q=X$ for every $x\in X$ and $t\ge c$. 
\end{lem}

\begin{proof} By Chow's Lemma and Hironaka's
resolution of singularities, there exists a proper birational morphism $\mu\colon X'\to X$ such that 
$X'$ is smooth and projective. Let $A'$ be very ample on $X'$. Let $n\ge 1$ and $0\ne s\in \Gamma(X,\cO_X(nL))$. 
Choose a point $x'\in \mu^{-1}(x)$. Then $0\ne \mu^*s\in \Gamma(X',\cO_{X'}(n\mu^*L))$ and
$$
\ord_x(s)\le \ord_{x'}(\mu^*s)\le n(\mu^*L\cdot A'^{d-1}),
$$
where the last inequality follows from Lemma~\ref{bm}.
Thus $\Bs|\cI_x^{t+}L|_\Q=X$ for $t\ge (\mu^*L\cdot A'^{d-1})$. 
\end{proof}

Recall that the codimension at $x$ of a closed subset $x\in Y\subseteq X$ is the smallest
codimension of an irreducible component of $Y$ passing through $x$. Since $X$ is irreducible
of dimension $d$, $\codim_x(Y\subseteq X)=d-\dim_x(Y)$. When the ambient space $X$ is fixed,
we will drop it from notation.

\begin{defn}
For $i\ge 1$, the {\em $i$-th successive minimum of $L$ at $x$} is defined by
$$
\epsilon_i(L,x)=\inf\{t\ge 0; \codim_x \Bs|\cI_x^{t+} L|_\Q < i  \}.
$$
\end{defn}

The definition makes sense, since $\codim_x \Bs|\cI_x^{t+} L|_\Q =0$ for $t\gg 0$. 
The infimum is a minimum, by Lemma~\ref{ug}. Note that 
$\codim_x \Bs|\cI_x^{t+} L|_\Q $ is strictly less than $i$ for $t\ge \epsilon_i(L,x)$, and 
at least $i$ for $0\le t<\epsilon_i(L,x)$. We obtain a chain of real numbers
$$
\epsilon_1(L,x)\ge \cdots\ge \epsilon_d(L,x)\ge \epsilon_{d+1}(L,x)=0.
$$
Note that $\codim_x \Bs|\cI_x^{t+}L|_\Q=0$ for $t\ge \epsilon_1(L,x)$, and 
$\codim_x \Bs|\cI_x^{t+}L|_\Q=i$ if and only if $\epsilon_i(L,x)>t\ge \epsilon_{i+1}(L,x)$.

\begin{rem}
One may similarly define successive minima $\epsilon_i(R,x)$ in points $x\in X$, 
associated to a graded subalgebra $R\subseteq R(L)$. For example, when $R$ is the image of the restriction 
$R(X',L)\to R(X,L'|_X)$, where $X\subset X'$ is a closed embedding and $L'$ is a Cartier divisor 
on $X'$. The successive minima of such subalgebras appear in Lemma~\ref{des}, for example.
\end{rem}

\begin{lem}
Let $L,L'$ be Cartier divisors on $X$, let $x\in X$ be a closed point and $1\le i\le d$.
\begin{itemize}
\item[a)] $\epsilon_i(qL,x)=q\epsilon_i(L,x)$ for every integer $q\ge 1$.
\item[b)] If $L\sim_\Q L'$, then $\epsilon_i(L,x)=\epsilon_i(L',x)$.
\item[c)] Suppose $\epsilon_i(L,x)>0$ and $\epsilon_i(L',x)>0$. Then
$\epsilon_i(L+L',x)\ge \epsilon_i(L,x)+\epsilon_i(L',x)$.
\end{itemize}
\end{lem}

\begin{proof} Property a) follows from $\Bs|\cI_x^{t+}(qL)|_\Q=\Bs|\cI_x^{\frac{t}{q}+} L |_\Q$,
and b) from $\Bs|\cI_x^{t+} L |_\Q=\Bs|\cI_x^{t+} L' |_\Q$ for all $t\ge 0$.

c) Let $0\le t<\epsilon_i(L,x)$ and $0\le t'<\epsilon_i(L',x)$. Then both
$\Bs|\cI_x^{t+}L|_\Q$ and $\Bs|\cI_x^{t'+}L'|_\Q$ have codimension at $x$ at least $i$. Since 
$$
\Bs|\cI_x^{(t+t')+}(L+L')|_\Q\subseteq \Bs|\cI_x^{t+}L|_\Q \cup \Bs|\cI_x^{t'+}L'|_\Q,
$$
the codimension at $x$ of the left hand side is at least $i$. Therefore $t+t'<\epsilon_i(L+L',x)$.
Letting $t$ and $t'$ approach $\epsilon_i(L,x)$ and $\epsilon_i(L',x)$, respectively, we obtain the claim.
\end{proof}

In particular, we may define the $i$-th successive minimum of a $\Q$-Cartier divisor $L$
at a point $x\in X$ to be $\epsilon_i(L,x)=\frac{1}{q}\epsilon_i(qL,x)$, where $q$ is a positive 
integer such that $qL$ is Cartier.

\begin{lem}
Let $f\colon X' \to X$ be a proper contraction of algebraic varieties, which is smooth
at a point $x'\in X'$. Let $L$ be a Cartier divisor on $X$. Then 
$\epsilon_i(L,f(x'))=\epsilon_i(f^*L,x')$ for all $i$.
\end{lem}

\begin{proof} Denote $L'=f^*L$ and $x=f(x')$. Since $f$ is a contraction, $f^*$ induces isomorphisms 
$
\Gamma(nL)\isoto \Gamma(nL') \ (n\ge 1).
$
Since $f^*$ maps $\cI_x^p$ into $\cI_{x'}^p$, it induces injective homomorphisms
$$
\Gamma(\cI_x^p(qL))\to \Gamma(\cI_{x'}^p(qL')) \ (p,q\ge 1).
$$
These are also surjective if $f$ is smooth at $x'$, since in this case the order of a section $s$
at $x$ coincides with the order of $f^*s$ at $x'$. We obtain isomorphisms
$$
f^*\colon \Gamma(\cI_x^p(qL))\isoto \Gamma(\cI_{x'}^p(qL')) \ (p,q\ge 1).
$$
Therefore $\Bs|\cI_{x'}^{t+}L'|_\Q=f^{-1}(\Bs|\cI_x^{t+}L|_\Q)$. Since $f$ is smooth at $x'$,
we obtain 
$$
\codim_{x'} \Bs|\cI_{x'}^{t+}L'|_\Q=\codim_x \Bs|\cI_x^{t+}L|_\Q.
$$
Therefore $\epsilon_i(L,x)=\epsilon_i(f^*L,x')$ for all $i$.
\end{proof}

\begin{lem}\label{er} $\epsilon_i(L,x)>0$ if and only if there exist integers $p,q\ge 1$ such that 
$\codim_x \Bs|\cI_x^p(qL)| \ge i$. Moreover, in this case we have
$$
\epsilon_i(L,x)=\sup\{\frac{p}{q}; p,q\ge 1, \codim_x \Bs|\cI_x^p(qL)| \ge i \}.
$$
\end{lem}

\begin{proof} Denote $\epsilon_i=\epsilon_i(L,x)$. 

Step 1: $\codim_x\Bs|\cI_x^p(qL)|< i$ for every $p,q\ge 1$ such that $p>q\epsilon_i$.
Indeed, $\Bs|\cI_x^p(qL)|$ contains $\Bs|\cI_x^{\epsilon_i+} L|_\Q$, which has codimension at $x$ strictly less than $i$.

Step 2: Let $0\le t<\epsilon_i$. Then there exist $p,q\ge 1$ with $t<\frac{p}{q}\le \epsilon_i(L,x)$ 
and $\codim_x \Bs|\cI_x^p(qL)|\ge i$.
Indeed, we have $\codim_x \Bs|\cI_x^{t+} L|_\Q\ge i$. There exists $r\ge 1$ such that 
$\Bs|\cI_x^{t+} L|_\Q=\Bs|\cI_x^{\lfloor rt\rfloor+1}(rL)|$. Let $q=r$ and $p=\lfloor rt\rfloor +1$.
By construction, $t<\frac{p}{q}$. By the first step, $\frac{p}{q}\le \epsilon_i$.

If $\epsilon_i>0$, there exist $p,q\ge 1$ such that $\codim_x \Bs|\cI_x^p(qL)|\ge i$, by Step 2.
Conversely, the latter implies $\epsilon_i\ge \frac{p}{q}>0$, by Step 1. The two steps
give that $\epsilon_i$ is the supremum after all such $\frac{p}{q}$.
 \end{proof}

\begin{lem} $\epsilon_1(L,x)=0$ if and only if one of the following holds:
\begin{itemize}
\item[a)] $\kappa(L)=-\infty$, or 
\item[b)] $\kappa(L)=0$ and every $0\ne s\in \Gamma(X,\cO_X(nL))\ (n\ge 1)$ is invertible at $x$.
\end{itemize}
\end{lem}

\begin{proof}
Suppose a) or b) holds, that is $\Gamma(\cI_x(nL))=0$ for $n\ge 1$. Therefore $\Bs|\cI_x^{0+} L|_\Q=X$. 
We obtain $\epsilon_1(L,x)=0$.

Conversely, suppose $\epsilon_1(L,x)=0$. We may assume we are not in case a), that is 
$\kappa(L)\ge 0$. Let $n\ge 1$ such that $\Gamma(nL)\ne 0$. If $\dim_k \Gamma(nL)\ge 2$,
then $\Gamma(\cI_x(nL))\ne 0$. This implies $\epsilon_1(L,x)\ge \frac{1}{n}$, a contradiction.
Therefore $\dim_k \Gamma(nL)=1$. Since again $\Gamma(X,\cI_x(nL))=0$, we have 
$\Gamma(nL)=ks_n$ with $s_n(x)\ne 0$.
\end{proof}

In particular, $\epsilon_1(L,x)$ is zero if $\kappa(L)<0$, and otherwise equals 
$$
\sup\{\frac{\ord_x(s)}{n};0\ne s\in \Gamma(nL), n\ge 1 \}.
$$

We call $\epsilon_1(L,x)$ the {\em width of $L$ at the point $x$}, also denoted by $\width_x(L)$. 
If $x\in X$ is a smooth point, then $\epsilon_1(L,x)$ coincides with the width of $L$ at the 
geometric valuation induced by the exceptional divisor of the blow-up of $X$ at $x$
(see~\cite[Section 2]{Amb16}).

By Lemma~\ref{er}, $\epsilon_d(L,x)>0$ if and only if there exist integers $p,q\ge 1$ such that
$\Bs|\cI_x^p(qL)|=\{x\}$ near $x$, and in this case, $\epsilon_d(L,x)$ is the supremum of 
$\frac{p}{q}$ after all such $p,q$.
We call $\epsilon_d(L,x)$ the {\em Seshadri constant of $L$ at $x$}, since we will see 
later (Proposition~\ref{nef}) 
that in the classical setting, it coincides with the Seshadri constant introduced by Demailly.

\begin{lem} $\epsilon_d(L,x)>0$ if and only if there exists $n\ge 1$ such that $x\notin \Bs|nL|$ and the 
induced rational map $\phi=\phi_{|nL|}\colon X\dashrightarrow X_n\subseteq |nL|$ 
satisfies $\phi^{-1}\phi(x)=\{x\}$ near $x$. In particular, $L$ is big.
\end{lem}

\begin{proof}
Suppose $\epsilon_d(L,x)>0$. Then there exists $p,q\ge 1$ such that 
$\Bs|\cI_x^p(qL)|=\{x\}$ near $x$. By Lemma~\ref{ibp}, there exists $l\ge 1$ such
that $\Bs|qlL|$ is empty in a neighborhood of $x$. Denote $n=ql$. Let $\phi=\phi_{|nL|}\colon X\dashrightarrow X_n\subseteq |nL|$
be the induced rational map. It is regular on the open dense subset $U=X\setminus \Bs|nL|$, which contains $x$.
Moreover, $\Bs|\cI_x(nL)|\cap U=\phi^{-1}\phi(x)\cap U$. Then 
$$
0=\dim_x\Bs|\cI_x(nL)|=\dim_x\phi^{-1}\phi(x)\ge \dim X-\dim X_n.
$$
Therefore $\phi^{-1}\phi(x)=\{x\}$ near $x$, and $\dim X_n=d$. In particular, $L$ is big.

Conversely, suppose $\phi=\phi_{|nL|}\colon X\dashrightarrow X_n\subseteq |nL|$ is regular near $x$
and $\phi^{-1}\phi(x)=\{x\}$ near $x$. Then $\Bs|\cI_x(nL)|$ equals $\{x\}$ near $x$. Therefore $\epsilon_d(L,x)\ge \frac{1}{n}>0$.
\end{proof}

The following lemma is a Successive minima version of~\cite[Proposition 6.4]{ELMNP}

\begin{lem}\label{moap}
Let $X$ be normal. Let $x\notin \Bs|L|_\Q$. For $n\ge 1$ such that $x\notin \Bs|nL|$, let 
$\phi_{|nL|}\colon X\dashrightarrow X_n\subseteq |nL|$ be the induced rational map. Let
$\mu_n\colon Y_n\to X$ be the normalization of the graph of $\phi_{|nL|}$. 
In the mobile-fixed decomposition $|\mu_n^*(nL)|=|M_n|+F_n$, the mobile part is base point free.
Moreover, $\mu_n$ is an isomorphism over a neighborhood of $x$. Then 
$\epsilon_i(L,x)$ is the supremum of $\frac{\epsilon_i(M_n,x)}{n}$, after all such $n$.
\end{lem}

\begin{proof}
Since $x\notin F_n\in |n\mu_n^*L-M_n|$, we have $\Bs|\cI_x^{t+}(n\mu^*L)|_\Q\subseteq \Bs|\cI_x^{t+}M_n|_\Q$
near $x$. Therefore $\epsilon_i(M_n,x)\le \epsilon_i(n\mu^*L,x)=n\epsilon_i(L,x)$.

Suppose now $t<\epsilon_i(L,x)$. There exist integers $p,q\ge 1$ with $p>qt$ and 
$\Bs|\cI_x^{t+}L|_\Q=\Bs|\cI_x^p(qL)|$. The identifications 
$\Gamma(Y_q,\cO_{Y_q}(M_q))=\Gamma(Y_q,\cO_{Y_q}(\mu_q^*qL))=\Gamma(X,\cO_X(qL))$
induce an identification 
$\Gamma(Y_q,\cI^p_{x\in Y_q}(M_q))=\Gamma(X,\cI_x^p(qL))$.
Therefore $\codim_x\Bs|\cI_x^p(M_q)|\ge i$.
We obtain $p\le \epsilon_i(M_q,x)$, that is $t< \frac{p}{q}\le \frac{\epsilon_i(M_q,x)}{q}$.

The supremum is in fact a limit.
\end{proof}


\subsection{Successive minima at a very general point}

Contrary to the usual Seshadri constant, $\epsilon_i(L,x)$ is not lower semicontinuous with respect to $x \in X$ in general.
For example, $\epsilon_1(L,x)$ can be upper semicontinuous.
But one can show a weaker property, that is $\epsilon_i(L,x)$ is independent of the choice of a very general point $x$.

\begin{defnprop} There exists a countable intersection of open dense subsets $V\subseteq X$ such that
the correspondence $V\ni x\mapsto \epsilon_i(L,x)$ is constant. The common value does not depend
on $V$, and is denoted by $\epsilon_i(L)$.
\end{defnprop}

\begin{proof}
Let $p_1,p_2\colon X\times X\to X$ be the natural projections, let $\delta\colon X\to X\times X$ be the diagonal
embedding, let $\Delta\subset X\times X$ be the diagonal. Let $p,q\ge 1$ be integers. 
Let $B^{pq}$ be the locus where the composition
$
p_2^*{p_2}_* \cI_\Delta^p(qp_1^*L) \to \cI_\Delta^p(qp_1^*L)\subset \cO(qp_1^*L)
$ 
is not surjective. 
By the definition, $B^{pq}$ contains $\Delta$. There exists an open dense subset $U^{pq}\subseteq X$ such that 
the following properties hold:
\begin{itemize}
\item[1)] ${p_2}_*(\cI_\Delta^p(qp_1^*L))\otimes k(x)\isoto \Gamma(\cI_x^p(qL))$ for every $x\in U^{pq}$. 
\item[2)] If $Y$ is an irreducible component of $B^{pq}$ which contains $\Delta$, then $p_2\colon Y\to X$
is flat over $U^{pq}$. In particular, $Y\cap p_2^{-1}(x)$ is equi-dimensional, of dimension $\dim Y-d$,
for every $x\in U^{pq}$.
\item[3)] $\delta(U^{pq})$ intersects only the irreducible components of $B^{pq}$ which contain $\Delta$.
\end{itemize}
Indeed, conditions 1) and 2) are open dense in $X$. For 3), we remove from $X$ the $\delta$-preimage of the 
trace on $\Delta$ of the irreducible components of $B^{pq}$ which do not contain $\Delta$. 
By 1), the natural inclusion $\Bs|\cI_x^p(qL)| \subseteq B^{pq}\cap p_2^{-1}(x)$ is an equality for every $x\in U^{pq}$.
By 2) and 3), we have
$$
\dim_{\delta(x)}B^{pq}\cap p_2^{-1}(x)=\codim(\Delta\subseteq B^{pq})\ \forall x\in U^{pq}.
$$
Let $V=\cap_{p,q\ge 1}U^{pq}$. It is dense in $X$, a countable intersection of open subsets of $X$.
We claim that $\epsilon_i(L,\cdot)$ is constant for $x\in V$. Suffices to show that if $x\in V$, $t\ge 0$ and $\epsilon_i(L,x)>t$,
then $\epsilon_i(L,y)>t$ for every $y\in V$. Indeed, $\epsilon_i(L,x)>t$ implies that there exists $\frac{p}{q}>t$
such that $\codim_x\Bs|\cI_x^p(qL)|\ge i$. That is $\codim_{\delta(x)}B^{pq}\cap p_2^{-1}(x)\ge i$. From above, this
is equivalent to $\codim(\Delta\subseteq B^{pq})\ge i$. Arguing backwards, we see that $\codim_y\Bs|\cI_y^p(qL)|\ge i$
for every $y\in V$. Therefore $\epsilon_i(L,y)\ge \frac{p}{q}>t$ for every $y\in V$.
\end{proof}

\begin{lem}\label{lb}
Suppose $|nL|\ne \emptyset$ and $\dim \phi_{|nL|}(X)\ge i$. Then $\epsilon_i(L)\ge \frac{1}{n}$.
\end{lem}

\begin{proof}
Let $U_n\subseteq X$ be the open dense subset on which $\phi=\phi_{|nL|}\colon X\dashrightarrow X_n\subseteq |nL|$ is regular. We have 
$
\Bs|\cI_x(nL)|\cap U_n=\phi^{-1}\phi(x)\cap U_n
$ 
for every $x\in U_n$. We have $\dim_x\phi^{-1}\phi(x)\ge \dim X-\dim X_n$ for all $x\in U_n$, and 
equality holds for an open dense subset $U'_n\subseteq U_n$. Let $x\in U'_n$. Then 
$$
\codim_x \Bs|\cI_x(nL)|=\codim_x \phi^{-1}\phi(x)=\dim X_n\ge i.
$$
By Lemma~\ref{er}, $\epsilon_i(L,x)\ge \frac{1}{n}$.
\end{proof}

\begin{lem}\label{ek}
$\epsilon_i(L)>0$ if and only if $\kappa(L)\ge i$.
\end{lem}

\begin{proof}
Suppose $\kappa(L)\ge i$. There exists $n\ge 1$ such that $|nL|\ne\emptyset$, and 
if $X\dashrightarrow X_n\subseteq |nL|$ is the induced rational map,
then $\dim X_n=\kappa(L) \ge i$. By Lemma~\ref{lb}, $\epsilon_i(L)\ge \frac{1}{n}$.

Suppose $\kappa(L)< i$. Suppose $\Gamma(nL)\ne 0$, let
$\phi_n\colon X\dashrightarrow X_n\subseteq |nL|$ be the induced rational map. For general $x\in X$ we have
$$
\codim_x\Bs|\cI_x(nL)|=\codim_x \phi_n^{-1}\phi_n(x)=\dim X_n< i.
$$
Therefore $\epsilon_i(L,x)< \frac{1}{n}$. Letting $n$ approach infinity, we obtain $\epsilon_i(L)=0$.
\end{proof}

Let $\kappa(L)=\kappa$. If $\kappa=0$, then $\epsilon_i(L)=0$ for all $i$. If $\kappa \ge 1$, we obtain 
$$
\epsilon_1(L)\ge \cdots\ge \epsilon_\kappa(L)>0=\epsilon_{\kappa+1}(L).
$$

\begin{lem}\label{LL'}
Suppose $\kappa(L'-L)\ge 0$. Then $\epsilon_i(L)\le \epsilon_i(L')$ for all $i$.
\end{lem}

\begin{proof}
We may suppose $|L'-L|$ is not empty. Choose $C\in |L'-L|$. Then $\Bs|\cI_x^p(qL')|\subseteq  \Bs|\cI_x^p(qL)| \cup \Supp C$.
Therefore $\Bs|\cI_x^{t+}L'|_\Q \subseteq  \Bs|\cI_x^{t+}L|_\Q \cup \Supp C$. For $x\in X\setminus C$, we obtain
$$
\codim_x \Bs|\cI_x^{t+}L'|_\Q \ge \codim_x \Bs|\cI_x^{t+}L|_\Q.
$$
By definition, $\epsilon_i(L,x)\le \epsilon_i(L',x)$ for all $i$.
\end{proof}

\begin{lem}\label{li}
Let $\epsilon_i(L)>0$, $\kappa(L')\ge 0$, and $|lL-l'L'|\ne \emptyset$ for some integers $l,l'\ge 1$.
Then $$\lim_{n\to \infty}\frac{\epsilon_i(nL+L')}{n}=\epsilon_i(L).$$
\end{lem}

\begin{proof}
Since $\kappa(L')\ge 0$, we have $\epsilon_i(nL+L')\ge n \epsilon_i(L)$. Let $C\in |lL-l'L'|$. For $n\ge 1$, we have 
$$
(l'n+l)L\sim l'(nL+L')+C.
$$
Therefore $(l'n+l)\epsilon_i(L)\ge l'\epsilon_i(nL+L')$ by Lemma~\ref{LL'}. We obtain 
$$
0\le \epsilon_i(nL+L')-n \epsilon_i(L) \le \frac{l\epsilon_i(L)}{l'}.
$$
Dividing by $n$ and letting $n$ approach $+\infty$, we obtain the claim.
\end{proof}

\begin{prop} Suppose $X$ is projective. Let $L\equiv L'$ be numerically equivalent big Cartier divisors on $X$. 
Then $\epsilon_i(L)=\epsilon_i(L')$ for every $i$.
\end{prop}

\begin{proof}
Let $A$ be an ample divisor on $X$. Since $(nL+A)-nL'$ is ample, $\epsilon_i(nL+A)\ge n\epsilon_i(L')$.
By Lemma~\ref{li}, we obtain $\epsilon_i(L)\ge \epsilon_i(L')$. The other inequality holds by the same argument.
\end{proof}

In particular, if $X$ is projective and $L$ is big, $\epsilon_1(L)$
coincides with the invariant $m(L)$ introduced by Nakamaye~\cite[page 2]{Nak03b}.

We may rephrase~\cite[Proposition 2.3]{EKL} and~\cite[Lemma 1.3]{Nak05} as follows:

\begin{lem}\label{ve} 
Let $X$ be a proper algebraic variety, let $L$ be a Cartier divisor, let $t\ge 0$.
Let $x\in X$ be a very general point and $Z\subseteq \Bs|\cI_x^{t+} L |_\Q$ an irreducible 
component containing $x$.
Then for integers $p,q\ge 1$ and $s\in \Gamma(\cI_x^p(qL))$, we have $\ord_Z(s)\ge p-qt$.
\end{lem}

\begin{proof} Let $p_1,p_2\colon X\times X\to X$ be the natural projections, let $\delta\colon X\to X\times X$ be the
diagonal embedding, let $\Delta\subset X\times X$ be the diagonal. 

Let $p,q\ge 1$ be integers with $p>qt$. 
Let $B^{pq}$ be the locus where the composition
$
p_2^*{p_2}_* \cI_\Delta^p(qp_1^*L) \to \cI_\Delta^p(qp_1^*L)\subset \cO(qp_1^*L)
$ 
is not surjective. There exists an open dense subset $X^{pq}\subseteq X$ such that 
$$
{p_2}_*(\cI_\Delta^p(qp_1^*L))\otimes k(x)\isoto \Gamma(X,\cI_x^p(qL)) \ \forall  x \in X^{pq}.
$$
We deduce that $\Bs|\cI_x^p(qL)|\subseteq B^{pq}\cap p_2^{-1}(x)$ for all $x\in X$, and equality holds for $x\in X^{pq}$.

Define $B=\cap_{p> qt}B^{pq}$, a closed subset of $X\times X$. Denote $X^0=\cap_{p> qt}X^{pq}$,
a countable intersection of open dense subsets of $X$. Then 
$\Bs|\cI_x^{t+} L|_\Q \subseteq B\cap p_2^{-1}(x)$ for all $x\in X$, and equality holds for $x\in X^0$.

We have $\Delta\subseteq B$, so at least one irreducible component of $B$ contains $\Delta$. 
Let $B'$ be the union of irreducible components of $B$ which do not contain $\Delta$. 
Then $\delta^{-1}(\Delta\cap B')$ is a proper closed subset of $X$. 
Its complement $U\subseteq X$ is an open dense subset. For every $x\in U$, an irreducible component of $B$ which
contains $\delta(x)$ must also contain $\Delta$. We may further shrink $U$, so that $U$ is smooth.

Let $x\in U\cap X^0$. Let $Z$ be an irreducible component of $\Bs|\cI_x^{t+} L|_\Q$ which contains $x$,
let $0\ne s\in \Gamma(\cI_x^p(qL))$.
There exists an irreducible component $\tilde{Z}$ of $B$, which contains $\delta(x)$, such that $Z\subseteq \tilde{Z} \cap p_2^{-1}(x)$.
Since $\tilde{Z}$ contains $\delta(x)$ and $x\in U$, $\tilde{Z}$ contains $\Delta$.
Since $x\in X^{pq}$, there exists $\tilde{s}\in \Gamma(X\times T,\cI_\Delta^p(q p_1^*L))$ with 
$\tilde{s}|_{p_2^{-1}(x)}=s$, where $T$ is a neighborhood of $x$ in $U$, which we may suppose affine and smooth.
Let $m=\ord_{\tilde{Z}}\tilde{s}$. Since $\tilde{Z}$ contains $\Delta$, $p_1\colon \tilde{Z}\to X$ is dominant. 
Let $D^m$ be a general differential operator of order at most $m$ on $T$. Then 
$D^m\tilde{s}\in \Gamma(X\times T,\cI_\Delta^{p-m}(q p_1^*L))$ does not vanish along $\tilde{Z}\cap X\times T$,
by the argument of~\cite[Proposition 2.3]{EKL} (the argument goes through even if $X$ is singular, since 
$T$ is affine smooth and $X\times T$ is smooth at the general point of $\Delta \cap X\times T$).
Therefore $\tilde{Z}\cap X\times T$ is not contained $B^{p-m,q}$. Therefore $p-m\le qt$.

We conclude $\ord_Z(s)=\ord_Z(\tilde{s}|_{p_2^{-1}(x)})\ge \ord_Z(\tilde{s})\ge \ord_{\tilde{Z}}(\tilde{s})=m\ge p-qt$.
\end{proof}


\subsection{Seshadri constant as $d$-th minimum}

Recall our notation: $X$ is a proper algebraic variety of dimension $d$, $L$ is a Cartier divisor on $X$, and 
$x\in X$ is a closed point.

\begin{prop}[Seshadri criterion]
The divisor $L$ is ample if and only if $\epsilon_d(L,x)>0$ for any $x \in X$.
\end{prop}

\begin{proof}
Suppose $L$ is ample. There exists $n\ge 1$ such that $nL$ is very ample. That is,
$\cI_x(nL)$ is generated by global sections, for every point $x\in X$. Then 
$\Bs|\cI_x(nL)|=\{x\}$, hence $\epsilon_d(L,x)\ge \frac{1}{n}$ for every $x$.

Conversely, suppose $\epsilon_d(L,x)>0$ for any $x \in X$. To show that $L$ is ample,
by Nakai-Moishezon's criterion, suffices to show that for every irreducible subvariety $Y\subseteq X$,
of dimension $r\ge 1$, we have $(L|_Y)^r>0$. Note that $0<\epsilon_d(L,x)\le \epsilon_r(L|_Y,x)$ for every 
$x\in Y$. 
If $r<d=\dim X$, we conclude by induction that $L|_Y$ is ample, hence $(L|_Y)^r>0$.
It remains to deal with the case $Y=X$. Since $\epsilon_d(L,x)$ is positive at a very general point,
$L$ is big by Lemma~\ref{ek}. Since $L$ is also nef by induction, we deduce $(L^d)>0$.
\end{proof}

\begin{prop}\label{nef}
Suppose $X$ is projective and $L$ is nef.
Then 
$$
\epsilon_d(L,x) = \inf_{C\ni x} \frac{(L\cdot C)}{\mult_x(C)}  
= \sup\{0\} \cup \{\frac{p}{q}; p,q\ge 1,  \Gamma(qL)\to \cO_x/\cI_x^{p+1} \text{ surjective} \}.
$$
\end{prop}

Denote by $\epsilon^s(L,x)$ the infimum after all integral curves passing through $x$, and by
$\epsilon^j (L,x)$ the supremum in the claim. 
First, we show the following lemma:

\begin{lem}\label{ineq}
For a Cartier divisor $L$ on a proper $X$,
\begin{itemize}
\item[a)] $\epsilon_d(L,x)\ge \epsilon^j(L,x)$,
\item[b)] if $L$ is nef, $\epsilon^s(L,x)\ge \epsilon_d(L,x)\ge \epsilon^j(L,x)$.
\end{itemize}
\end{lem}

\begin{proof}
a) Let $p,q\ge 1$ such that 
$\Gamma(qL)\to \cO_x/\cI_x^{p+1}$ is surjective. Equivalently, $\cI_x^i(qL)$ is globally generated 
at $x$ for every $0\le i\le p$. Then $\Bs|\cI_x^p(qL)|=\{x\}$ near $x$. Therefore $\frac{p}{q}\le \epsilon_d(L,x)$.
Taking the supremum after all $\frac{p}{q}$, we obtain $\epsilon_d(L,x)\ge \epsilon^j(L,x)$.

b) Suppose $L$ is nef.
If $ \epsilon_d(L,x) =0$, $\epsilon^s(L,x)\ge \epsilon_d(L,x)$ is clear.
Assume $ \epsilon_d(L,x) > 0$ and let $0 <t<\epsilon_d(L,x)$. Then $\Bs|\cI_x^{t+}L|_\Q=\{x\}$
near $x$. Let $C\subseteq X$ be an integral curve passing through $x$. Since $C$ is not contained
in $\Bs|\cI_x^{t+}L|_\Q$, there exists $n\ge 1$ and $D\in |nL|$ such that $\ord_x(D)> nt$ and $C$ is not
contained in the support of $D$. Then 
$$
n(L\cdot C)=(D\cdot C)\ge \ord_x(D)\cdot \mult_x(C)> nt \mult_x(C).
$$
Therefore $t \le \epsilon^s(L,x)$. Taking the supremum after all such $t$, we obtain $\epsilon^s(L,x)\ge \epsilon_d(L,x)$.
\end{proof}

To prove Proposition~\ref{nef}, it is enough to show $\epsilon^s(L,x) \le \epsilon^j(L,x)$ by Lemma~\ref{ineq}.
In fact, $\epsilon^s(L,x) = \epsilon^j(L,x)$ is well known at least for ample $L$ and smooth $x \in X$ (see \cite[Theorem 5.1.17]{Lazbook}).

\begin{proof}[Proof of Proposition~\ref{nef}]
Let $f\colon Y\to X$ be the blow-up of $X$ at $x$, with
exceptional divisor $E$. Note that $f^*L-\epsilon^s(L,x)E$ is nef, and $\epsilon^s(L,x) \ge 0$ is maximal
with this property.

Let $A$ be an ample divisor on $X$.
The augmented base locus $\bB_{+}(L)$ of $L$ is defined by
$$
\bB_{+}(L) =\bigcap_{l \in \N} \Bs|lL- A|_{\Q},
$$
which does not depend on the choice of $A$.
By \cite{Nak00}, \cite{Bi},
$$
\bB_{+}(L) = \bigcup_{V} V
$$
holds,
where the union runs over the subvarieties $V \subseteq X$ such that $L|_V$ is not big.

First,
assume $x \in \bB_{+}(L)$.
Then there exists a subvariety $V \subseteq X$ containing $x$ such that $L|_V$ is not big.
Since $L$ is nef, we have $(L^{\dim V}.V)=0$.
Hence $\epsilon^s(L,x) =0$ follows from 
$$
0 \le  ((f^*L-\epsilon^s(L,x)E)^{\dim V'}. V') =(L^{\dim V}.V) - {\epsilon^s(L,x)}^{\dim V} \cdot \mult_x(V) = - {\epsilon^s(L,x)}^{\dim V} \cdot \mult_x(V) \le 0,
$$
where $V' \subseteq Y$ is the strict transform of $V$.
By Lemma~\ref{ineq} b), all three invariants equal to $0$.

So we may assume $x \not \in \bB_{+}(L)$.
If we replace $A$ with its suitable multiple,
we can take $l \geq 1$ and an effective Cartier divisor $D$ such that $lL \sim A+D$ and $x \not \in \Supp D$.
Since $A$ is ample and $-E$ is $f$-ample, there exists an integer $a\ge 1$
such that $af^*A-E$ is ample. 

By Lemma~\ref{ineq} b), it suffices to show $\epsilon^s(L,x)\le \epsilon^j(L,x)$, which follows from the following claim: 
if $p,q\ge 1$ are integers such that $p-1\le (q-la)\epsilon^s(L,x)$, then
$\Gamma(nqL)\to \cO_x/\cI_x^{np+1}$ is surjective for $n\gg 0$.
Indeed, the divisor
$$
f^*(qL-aD)-pE=(af^*A-E)+((q-la)f^*L-(p-1)E)
$$
is ample. By Serre vanishing, $H^1(Y,\cI_E(f^*(nqL-naD)-npE))=0$ for $n\gg 0$. 
We obtain an exact sequence
$$
\Gamma(Y,f^*(nqL-naD))\to \Gamma((np+1)E,\cO_{(np+1)E})\to H^1(Y,nqf^*L-naD-(np+1)E)=0.
$$
Since $n\gg 0$, the hypothesis of Lemma~\ref{lj} are satisfied. Therefore
$\Gamma(nqL-naD)\to \cO_x/\cI_x^{np+1}$ is surjective for $n\gg 0$. 
Since $D$ is effective and away from $x$, we deduce that $\Gamma(nqL)\to \cO_x/\cI_x^{np+1}$ is surjective.
Therefore $\frac{p}{q}\le \epsilon^j(L,x)$. Since $\frac{p}{q}$ can be arbitrary close to $ \epsilon^s(L,x)$, we obtain $\epsilon^s(L,x)\le \epsilon^j(L,x)$.
\end{proof}

The equalities in Proposition~\ref{nef} hold on proper $X$ if $L$ is semiample:

\begin{lem}\label{sa}
Suppose $X$ is proper and $L$ is semiample. Then 
$$
\epsilon_d(L,x) = \inf_{C\ni x} \frac{(L\cdot C)}{\mult_x (C)}  
= \sup \{0\} \cup \{\frac{p}{q}; p,q\ge 1,  \Gamma(qL)\to \cO_x/\cI_x^{p+1} \text{ surjective} \} .
$$
\end{lem}

\begin{proof}
There exists a proper contraction $f\colon X\to Y$, with $Y$ projective, such that 
$nL\sim f^*A$ for some  integer $n\ge 1$ and some ample divisor $A$ on $Y$. Let $\Exc(f)$
be the exceptional locus of $f$, that is the locus of points $x\in X$ such that $\dim_xf^{-1}f(x)>0$.

If $x\in \Exc(f)$, we see that $\epsilon^s(L,x)=\epsilon_d(L,x)=\epsilon^j(L,x)=0$. Suppose
$x\notin \Exc(f)$. Then $\dim_xf^{-1}f(x)=0$. By Lemma~\ref{icr}, $f$ is an isomorphism over a 
neighborhood of $f(x)$.
It follows that $\epsilon_d(nL,x)=\epsilon_d(A,f(x))$ and $\epsilon^*(nL,x)=\epsilon^*(A,f(x))$ for $*=s,j$. 
We are reduced to the ample case, so we conclude from Proposition~\ref{nef}.
\end{proof}

\begin{rem}
By Lemma \ref{sa}, the Seshadri constant of a semiample divisor $L$ at a point $x\in X$ defined by 
Demailly~\cite[Theorem 6.4]{De92} coincides with $\epsilon_d(L,x)$. 
Combining with Lemma~\ref{moap}, we deduce that if $X$ is normal, projective, $L$ is big, 
and $x$ is a point away from the augmented base locus of $L$,
then $\epsilon_d(L,x)$ coincides with the moving Seshadri constant $\epsilon_m(x,L)$ introduced by Nakamaye~\cite[Definition 0.4]{Nak03} (see also \cite[Section 6]{ELMNP}).
\end{rem}

The interpretation of the Seshadri
constant in terms of jets generation can be sharpened as follows: if $\epsilon_d(L,x)=\epsilon>0$, there 
exist constants $c_1(x)\ge 0$ and $c_2(x)\ge 0$ such that $\Gamma(qL)\to \cO_x/\cI_x^{p+1}$ is surjective for all integers
$p,q\ge 1$ such that $p\ge c_1(x)$ and $q\epsilon-p\ge c_2(x)$. This statement reduces again to the
ample case, which is the following Lemma.

\begin{lem}\label{lem_eff}
Suppose $L$ is ample. Let
$
\epsilon=\inf_{C\ni x}\frac{(L\cdot C)}{\mult_x(C)}
$
the Seshadri constant of $L$ at $x$. Then there exist constants $c_1,c_2\ge 0$ such that for every 
integers $p,q\ge 1$ with $p\ge c_1$ and $q\epsilon-p\ge c_2$, the jet map $\Gamma(qL)\to \cO_X/\cI_x^{p+1}$ 
is surjective.
\end{lem}

\begin{proof}
Let $f \colon Y\to X$ be the blow-up of $X$ at $x$, with exceptional divisor $E$. 
Since $L$ is ample and $-E$ is $f$-ample,
there exists an integer $a\ge 1$ such that $af^*L-E$ is ample. By Nakai's ampleness criterion, $qf^*L-pE$ is 
ample if and only if $q\epsilon>p>0$. We have 
$$
qf^*L-pE=m(af^*L-E)+(q-ma)f^*L-(p-m)E.
$$
For $(q-ma)\epsilon\ge p-m\ge 0$, the divisor $(q-ma)f^*L-(p-m)E$ is nef. By Fujita's extension of Serre vanishing
(see~\cite[Theorem 1.4.35]{Lazbook}), there exists an integer $m=m(\cO_Y,af^*L-E)$ such that $H^1(qf^*L-pE)=0$
for $(q-ma)\epsilon\ge p-m\ge 0$.

We have $\cI_x^{p+1}=f_*(\cI_E^{p+1})$ and $R^1f_*(\cI_E^{p+1})=0$ for $p\ge c'(x)$. Set $c_1=\max(c'(x),m-1)$
and $c_2=ma\epsilon-m+1$.

Let $p,q\ge 1$ be integers with $p\ge c_1$ and $q\epsilon-p\ge c_2$. Then $p\ge c'(x)$ and $(q-ma)\epsilon\ge p+1-m\ge 0$. 
The latter inequality implies that we have a surjection 
$$
\Gamma(qf^*L)\to \Gamma(\cO_{(p+1)E})\to 0.
$$
From the former inequality and Lemma~\ref{lj}, we deduce that $\Gamma(qL)\to \cO_X/\cI_x^{p+1}$ is surjective.
\end{proof}


\subsection{Successive minima in terms of jets generation}

Once again, $X$ is a proper algebraic variety, $L$ is a Cartier divisor on $X$, and $x\in X$ is a closed point.

The following proposition is proved in \cite[Proposition 6.6]{ELMNP} for the moving Seshadri constant on smooth $X$:

\begin{prop}\label{23}
We have $\epsilon_d(L,x)>0$ if and only if there exist $p,q\ge 1$ such that $\Gamma(qL)\to \cO_x/\cI_x^{p+1}$
is surjective. Moreover, in this case we have 
$$
\epsilon_d(L,x) = \sup\{\frac{p}{q}; p,q\ge 1,  \Gamma(qL)\to \cO_x/\cI_x^{p+1} \text{ surjective} \} . 
$$
\end{prop}

\begin{proof}
Denote by  $\epsilon^j(L,x)$ the supremum in the claim. 
We obtain $\epsilon_d(L,x)\ge \epsilon^j(L,x)$ by Lemma~\ref{ineq} a).
For the converse, 
we show the following claim:

\begin{claim}
If $\Bs|\cI_x^r(L)|=\{x\}$ near $x$ for some integer $r >0$, then $r \leq \epsilon^j(L,x) $.
\end{claim}

\begin{proof}
We use the notation in Lemma~\ref{rda}.
We identify $x$ with the point $g (f^{-1}(x) ) \in Y$.
Recall that
we have a natural homomorphism 
$\Gamma(\cI_Z^q(qL)) \rightarrow \Gamma (q A)$,
which is an isomorphism for $q \geq c(\cI_Z)$.

Since $\Bs|\cI_x^r(L)|=\{x\}$ near $x$,
$\Bs|\cI_x^r(A)|=\{x\}$ near $x \in Y$ and hence $ r \leq \epsilon_d(A,x)$ holds.
By applying Proposition~\ref{nef} to the ample divisor $A$,
$ \epsilon_d(A,x) =\epsilon^j(A,x) $ holds.
Hence, for any $0 < t <r $,
we can take $t < \frac{p}{q} < r$ such that $\Gamma(qA)\to \cO_{x,Y}/\cI_x^{p+1}$ is surjective.
We note that we can take arbitrary large such $p,q$ by Lemma~\ref{lem_eff}.
So we take such $p,q$ with $q \geq c(\cI_Z)$.

Since $q \geq c(\cI_Z)$, we have an  isomorphism $\Gamma(\cI_Z^q(qL)) \stackrel{\sim}{\rightarrow} \Gamma (q A)$.
Hence $\Gamma(\cI_Z^q(qL))\to \cO_{x,X}/\cI_x^{p+1}$ is surjective since
so is $\Gamma(qA)\to \cO_{x,Y}/\cI_x^{p+1}$.
Then $\Gamma(qL)\to \cO_{x,X}/\cI_x^{p+1}$ is also surjective since $x \not \in Z$,
and hence $t < \frac{p}{q} \leq \epsilon^j(L,x)$.
Letting $t$ approach $r$,
we obtain the claim.
\end{proof}

For any $0 \leq t' < \epsilon_d(L,x) $, 
we can take $t' < \frac{r}{s} < \epsilon_d(L,x)$ such that $\Bs|\cI_x^r(sL)|=\{x\}$ near $x$.
Applying the above claim to $sL$,
it holds that
$$
r \leq  \epsilon^j(sL,x) \leq s \cdot \epsilon^j(L,x),
$$
where the last inequality follows from the definition of $\epsilon^j$.
Hence $t' < \frac{r}{s} \leq \epsilon^j(L,x)$ and the proposition follows by letting $t'$ approach $\epsilon_d(L,x) $.
\end{proof}

\begin{lem}
Suppose $\epsilon_1(L,x)>0$. Then 
$$
\epsilon_1(L,x)=\sup\{\frac{p}{q}; p,q\ge 1, \Gamma(\cI_x^p(qL))\to  \cI_x^p/\cI_x^{p+1} \text{ non-zero} \}.
$$
\end{lem}

\begin{proof}
Note that $\Gamma(\cI_x^p(qL))\to \cI_x^p/\cI_x^{p+1}$ is non-zero if and only if there exists $s\in \Gamma(qL)$ such that
$\ord_x(s)=p$. In particular, $\epsilon_1(L,x)\ge \frac{p}{q}$. On the other hand, let $p,q\ge 1$ such that $\Gamma(\cI_x^p(qL))\ne 0$.
Let $p'\ge p$ be maximal such that $\Gamma(\cI_x^{p'}(qL))\ne 0$. There exists $s\in \Gamma(qL)$ such that $\ord_x(s)=p'$.
Therefore $\frac{p}{q}\le \frac{p'}{q}\le \epsilon_1(L,x)$. We conclude by Lemma~\ref{er}.
\end{proof}

\begin{lem}
Suppose $\epsilon_d(L,x)>0$. Then
$$
\epsilon_d(L,x) = \sup\{\frac{p}{q}; p,q\ge 1, \Gamma(\cI_x^p(qL))\to \cI_x^p/\cI_x^{p+1} \text{ surjective} \}.
$$
\end{lem}

\begin{proof}
The inequality $\le$ follows from Proposition~\ref{23}. For the opposite inequality, suppose that
$\Gamma(\cI_x^p(qL))\to \cI_x^p/\cI_x^{p+1}$ is surjective. By Nakayama's Lemma, $\cI_x^p(qL)$ 
is generated by global sections at $x$. Therefore $\Bs|\cI_x^p(qL)|=\{x\}$ near $x$. 
Therefore $p/q\le \epsilon_d(L,x)$.
\end{proof}


\section{Volume versus the successive minima}


The following well known statement may be considered as the analogue of 
Minkowski's first main theorem:

\begin{prop}
Let $X$ be proper of dimension $d$, let $L$ be a Cartier divisor on $X$, let $x\in X$ be a closed point. Then 
$\vol(L)\le \mult_x(X)\cdot \epsilon_1(L,x)^d$.
\end{prop}

\begin{proof}
Let $t>\epsilon_1(L,x)$ be a real number. Let $n\ge 1$. 
We have $\Gamma(\cI_x^{\lfloor nt\rfloor+1}(nL))=0$, that is the jet map
$$
\Gamma(nL)\to \cO_x/\cI_x^{\lfloor nt\rfloor+1}
$$
is injective. Therefore $h^0(nL)\le \dim_k(\cO_x/\cI_x^{\lfloor nt\rfloor+1})$. For $n\gg 0$,
the right hand side equals $P(\lfloor nt\rfloor)$, where $P(T)$ is a polynomial with leading term
$\frac{e}{d!}T^d$, where $e=\mult_x(X)$. Therefore 
$$
\limsup_{n\to \infty}\frac{h^0(nL)}{n^d/d!}\le e\cdot t^d.
$$
Letting $t$ approach $\epsilon_1(L,x)$, we obtain the claim.
\end{proof}

\begin{lem}\label{mil}
Let $X$ be a proper variety of dimension $d$, let $\epsilon_i=\epsilon_i(L,x)$ be the 
Seshadri successive minima of a Cartier divisor $L$ at a closed point $x$. Then 
$$
\mult_x(X)\cdot \prod_{i=1}^d\epsilon_i(L,x)\le \vol(L).
$$
In particular, $\mult_x(X)\cdot \epsilon_d(L,x)^d\le \vol(L)$.
\end{lem}

\begin{proof}
Suffices to show that for any real numbers $0\le t_i<\epsilon_i$, we have 
$$
\mult_x(X)\cdot \prod_{i=1}^dt_i< \vol(L).
$$
To prove this inequality, note first that there exist $D_1,\ldots,D_d\in |L|_\Q$ such that 
$\ord_x(D_i)>t_i$ and $\cap_{i=1}^dD_i=\{x\}$ near $x$. Indeed, there exists 
$D_1\in |L|_\Q$ with $\ord_x(D_1)>t_1$. Since $\codim_x \Bs|\cI_x^{t_2+}L|_\Q\ge 2$,
no irreducible component of $D_1$ through $x$ is contained in $\Bs|\cI_x^{t_2+}L|_\Q$.
There exists $D_2\in |L|_\Q$ with $\ord_x(D_2)>t_2$ and $\codim_x(D_1\cap D_2)=2$. 
Since $\codim_x \Bs|\cI_x^{t_3+}L|_\Q\ge 3$, no irreducible component of $D_1\cap D_2$ through 
$x$ is contained in $\Bs|\cI_x^{t_3+}L|_\Q$.
Therefore there exists $D_3\in |L|_\Q$ with $\ord_x(D_3)>t_3$ and $\codim_x(D_1\cap D_2\cap D_3)=3$.
Iterating the construction, we obtain the chain of divisors in $d$ steps.

There exists an integer $q\ge 1$ such that $qD_i$ are the zero divisors of some non-zero global sections
$s_i\in \Gamma(X,\cI_x^{\lfloor qt_i\rfloor+1}(qL))$. 
Let $Z  $ be the (possibly empty) subscheme $\cap_{i=1}^d qD_i \setminus \{x\} \subset X$.
Let $f\colon Y\to X$ be the 
blow-up of $X$ along $\cI_Z$, let $E$ be the exceptional divisor on $Y$. Then $f^*s_1,\ldots,f^*s_d$ become 
global sections of $\cO_Y(qf^*L-E)$, having no common zeroes near $E$. Therefore 
$\Bs|qf^*L-E|\subseteq \{x\}$ (we identify $x$ with a point of $Y$). By Lemma~\ref{ibp}, 
$M=qf^*L-E$ is semiample.

Set $D'_i=f^* qD_i - E \in |M | $.
By construction,
$\cap_{i=1}^d D'_i = \{x\} \subset Y$ as sets.
Hence we have
$$
(M^d)=(D'_1 \cdot D'_2  \cdots D'_d) = i (x,D'_1 \cdot D'_2  \cdots D'_d ;Y ) \geq \mult_x (Y) \cdot \prod_{i=1}^d (\lfloor qt_i\rfloor+1) > \mult_x(X)\cdot \prod_{i=1}^d qt_i
$$
by~\cite[page 233]{Fulton},
where $ i (x,D'_1 \cdot D'_2  \cdots D'_d ;Y )$ is the intersection multiplicity.
For $l\ge c(\cI_Z)$,
the natural homomorphism $\cI_Z^l\to f_*\cO_Y(-lE)$ is an isomorphism,
hence so is
$$
\Gamma(X,\cI_Z^l(lqL)) \stackrel{\sim}{\rightarrow} \Gamma(Y,\cO_Y(lM)).
$$
Therefore $q^d \vol(L)=\vol(qL)\ge \vol(M)=(M^d)>\mult_x(X)\cdot \prod_{i=1}^d qt_i$. Dividing out by $q$ and letting $t_i$ approach $\epsilon_i$, we obtain the claim. 
\end{proof}

\begin{cor}
Let $X$ be an algebraic variety of dimension $d$, let $x$ be a closed point and $L$ a Cartier divisor on $X$.
Then 
$$
\epsilon_d(L,x) \le \sqrt[d]{\frac{\vol(L)}{\mult_x(X)}} \le \epsilon_1(L,x).
$$
\end{cor}


\subsection{Minkowski's second main theorem for Seshadri successive minima}


\begin{lem}\label{gup}
Let $L$ be a Cartier divisor on a proper algebraic variety $X$.
Let $\epsilon_i=\epsilon_i(L)$ be the Seshadri successive minima of $L$ at a very general
point. Then 
$$
h^0(L) \le | \Z^d\cap \square(\epsilon_1,\ldots,\epsilon_d) | .
$$
\end{lem}

\begin{proof}  Let $x\in X$ be a very general point.

Step 1: Let $p\ge 0$ be an integer. Then 
$$
h^0(\cI_x^p(L))-h^0(\cI_x^{p+1}(L))\le | \cap_{i=2}^d \{\alpha\in \N^d(p); \alpha_i+\cdots+\alpha_d\le \epsilon_i \}|.
$$
Indeed, we have an exact sequence 
$$
0\to \Gamma(\cI_x^{p+1}(L)) \to \Gamma(\cI_x^p (L)) \stackrel{r}{\to} \cI_x^p/\cI_x^{p+1}\simeq \Gamma (\cO_{\bP^{d-1}}(p)).
$$
Therefore $h^0(\cI_x^p(L))-h^0(\cI_x^{p+1}(L))=\dim_k \im(r)$. We may suppose $\im(r)\ne 0$.
In particular,  $h^0(\cI_x^p(L))>0$, hence $p\le \epsilon_1$.
If $p\le \epsilon_d$, the desired inequality becomes $\dim \im(r)\le |\N^d(p)|$, which follows from the inclusion
$\im(r)\subseteq \Gamma (\cO_{\bP^{d-1}}(p))$. Therefore we may suppose $\epsilon_d<p \le \epsilon_1$.

Denote by $J$ the set of indices $1< i\le d$ such that $\epsilon_i<p$ and $\epsilon_i<\epsilon_{i-1}$.
Fix $i\in J$. Since $\epsilon_i<\epsilon_{i-1}$, $\Bs|\cI_x^{\epsilon_i+}L|_\Q$ has codimension $i-1$ at $x$.
Choose an irreducible component $Z^{i-1}$, of codimension $i-1$, passing through $x$. On the blow-up of $X$ at $x$, 
the proper transform of $Z^{i-1}$ has in common with the exceptional locus $E\simeq \bP^{d-1}$ at least one irreducible 
subvariety $W^{i-1}\subset \bP^{d-1}$, of codimension $i-1$. 
Let $s\in \Gamma(\cI_x^p (L))$. By Lemma~\ref{ve}, $\ord_{Z^{i-1}}(s)\ge p-\epsilon_i$. Therefore the image of $r$
is contained in $\{P\in\Gamma (\cO_{\bP^{d-1}}(p)); \ord_{W^{i-1}}(P)\ge p-\epsilon_i \}$. We conclude
$$
\im(r)\subseteq \cap_{i\in J} \{P\in\Gamma (\cO_{\bP^{d-1}}(p)); \ord_{W^{i-1}}(P)\ge p-\epsilon_i \}.
$$
By Proposition~\ref{post}, the right hand side has dimension at most 
$$
| \cap_{i\in J} \{\alpha \in \N^d(p); \alpha_i+\cdots+\alpha_d \le \epsilon_i \} |.
$$
From the definition of $J$, we see that 
$$
\cap_{i\in J} \{\alpha \in \N^d(p); \alpha_i+\cdots+\alpha_d \le \epsilon_i \}=
\cap_{i=2}^d \{\alpha \in \N^d(p); \alpha_i+\cdots+\alpha_d \le \epsilon_i \}.
$$
Therefore 
$
\dim_k \im(r)\le | \cap_{i=2}^d \{\alpha \in \N^d(p); \alpha_i+\cdots+\alpha_d \le \epsilon_i \} |.
$

Step 2: Let $p\ge 0$ be an integer. Then 
$$
h^0(L)-h^0(\cI_x^{p+1}(L))\le | \Z^d\cap \square(p,\epsilon_2,\ldots,\epsilon_d) |.
$$
Indeed, $h^0(L)-h^0(\cI_x^{p+1}(L))=\sum_{l=0}^p h^0(\cI_x^l (L))-h^0(\cI_x^{l+1}(L))$.
Applying Step 1 to each term of the sum, we obtain the desired inequality.

Step 3) Set $p=\lfloor \epsilon_1\rfloor$ in Step 2. 
Then $p+1>\epsilon_1$, so $h^0(\cI_x^{p+1}(L))=0$. We obtain
$$
h^0(L) \le | \Z^d\cap \square(\lfloor \epsilon_1\rfloor, \epsilon_2,\ldots,\epsilon_d) | .
$$
Since $\Z^d\cap \square(\lfloor t_1\rfloor,\ldots,\lfloor t_d\rfloor)=\Z^d\cap \square(t_1,\ldots,t_d)$, 
we obtain
$$
h^0(L) \le | \Z^d\cap \square(\epsilon_1,\ldots,\epsilon_d) | .
$$
\end{proof}

\begin{prop}\label{mir}
Let $L$ be a Cartier divisor on a proper algebraic variety $X$, with Iitaka dimension $\kappa(L)=\kappa\ge 1$.
Let $\epsilon_i=\epsilon_i(L)$ be the Seshadri successive minima of $L$ at a very general
point. Then 
$$
\limsup_{n\to \infty}\frac{h^0(nL)}{n^\kappa/\kappa!} \le \kappa!\cdot \vol  \square(\epsilon_1,\ldots,\epsilon_\kappa).
$$
In particular, if $L$ is big, $\vol(L) \le d!\cdot \vol  \square(\epsilon_1,\ldots,\epsilon_d)$.
\end{prop}

\begin{proof} Suppose $t_{\kappa+1}=\cdots=t_d=0$. Then $\square(t_1,\ldots,t_d)=\square(t_1,\ldots,t_{\kappa})\times 0$,
where $0$ is the origin in $\R^{d-\kappa}$. We have $\epsilon_1\ge \cdots \ge \epsilon_\kappa>0=
\epsilon_{\kappa+1}=\cdots=\epsilon_d$. Since $\epsilon_i(nL)=n\epsilon_i(L)$, we obtain
$$
h^0(nL)\le |\Z^\kappa \cap  \square(n\epsilon_1,\ldots,n\epsilon_\kappa)  | \ \forall n\ge 0
$$
As $n\to\infty$, the right hand side grows like $\vol  \square(\epsilon_1,\ldots,\epsilon_\kappa) \cdot n^{\kappa}+O(n^{\kappa-1})$. 
Therefore the claim holds.
\end{proof}

Proposition~\ref{mir} generalizes the following result of 
Nakamaye~\cite[Proof of Corollary 3]{Nak03b}: if $L$ is an ample
divisor on a smooth projective surface, then $(L^2)\le 2\epsilon_1(L)\epsilon_2(L)-\epsilon_2(L)^2$.

Combining Lemma~\ref{vol}, Propositions~\ref{mil} and~\ref{mir}, we obtain the following equivalent of Minkowski's
second theorem for successive minima:

\begin{thm}\label{m2}
Let $L$ be a big Cartier divisor on a $d$-dimensional proper algebraic variety $X$. Then 
$$
1 \le \frac{\vol(L)}{\prod_{i=1}^d\epsilon_i(L)}\le d! 
$$
So the volume of $L$ is essentially the product of the successive Seshadri minima of $L$ at a 
very general point. 
\end{thm}

\begin{rem}
If $(X,L)=(\bP^d,\cO(1))$, $\epsilon_i(L)=1$ for any $i$. Hence $ \frac{\vol(L)}{\prod_{i=1}^d\epsilon_i(L)} =1$.
On the other hand,
consider $(X,L)=((\bP^1)^d,\cO(w_1,\ldots,w_d))$, where $w_1\ge w_2\ge \cdots\ge w_d$ are positive integers.
Then $\epsilon_i(L)=\sum_{j=i}^d w_j$
and
$\vol(L)= d! \prod_{i=1}^dw_i $
as we will see in Example~\ref{product}.
Hence 
$$ \frac{\vol(L)}{\prod_{i=1}^d\epsilon_i(L)} =d! \prod_{i=1}^d \frac{w_i}{\sum_{j=i}^d w_j}$$ could be arbitrary close to $d!$
if we take $w_1\gg w_2\gg \cdots\gg w_d$.

Thus the lower and upper bounds in Theorem~\ref{m2} are sharp.
We also note that the upper bound is not attained for $d \geq 2$
since $ \vol  \square(\epsilon_1(L),\ldots,\epsilon_d(L)) < \prod_{i=1}^d\epsilon_i(L)$ by $ \square(\epsilon_1(L),\ldots,\epsilon_d(L)) \subsetneq [0,\epsilon_1(L)] \times  \square(\epsilon_2(L),\ldots,\epsilon_d(L))$.
\end{rem}


\section{Succesive minima on toric varieties}


For standard terminology on toric varieties, the reader may consult~\cite{Oda88}.
Let $X=T_N\emb(\Delta)$ be a proper toric variety (normal), of dimension $d$.
Let $L$ be a Cartier divisor on $X$. Modulo linear equivalence, we may suppose $L$
is torus invariant. Due to the torus action,  $\epsilon_i(L,\cdot)$ is constant on $T_N \subset X$. 
Therefore $\epsilon_i(L)=\epsilon_i(L,1)$, where $1$ denotes the unit of the torus $T_N$.
We will estimate $\epsilon_i(L,x)$ when either $x$ is a closed invariant point, or $x=1$.


\subsection{At a closed invariant point}


Let $x\in X$ be a closed invariant point. It corresponds to a top cone $\sigma\in \Delta(top)$,
and we obtain an open affine neighborhood $x\in U_\sigma=\Spec k[M\cap \sigma^\vee]$.
Denote by $S$ the semigroup $M\cap \sigma^\vee\setminus 0$, so that 
$I(x\in U_\sigma)=\oplus_{m\in S}k \cdot \chi^m$.
For $p\ge 1$, denote $S^{(p)}=\{s_1+\cdots+s_p;s_1,\ldots,s_p\in S\}$. Therefore 
$I(x\in U_\sigma)^p=\oplus_{m\in S^{(p)}}k \cdot \chi^m$.
Since $L$ is Cartier, there exists $u\in M$ such that $(\chi^u)+L|_{U_\sigma}=0$. We have 
$\square_L-u\subseteq \sigma^\vee$.

The subspace $\{s\in \Gamma(X,\cO_X(qL)); \ord_x(s)\ge p\}\subseteq  \Gamma(X,\cO_X(qL))$ 
is torus invariant. Therefore a basis over $k$ consists of monomials $\chi^m$ such that 
$m\in M\cap q\square_L$ and $m-qu\in S^{(p)}$.

Let $P=\Conv(S+\sigma^\vee)$ be the Newton polytope associated to $S\subset \sigma^\vee$.
Let $B$ be the closure of the complement $\sigma^\vee\setminus P$. Then $B$ is compact and
contains a relatively open neighborhood of the origin in $\sigma^\vee$. In particular, 
$\sigma^\vee=\cup_{t\ge 0}tB$. Note that $B$ may not be convex, but in case $x$ is a smooth point,
$B$ is a unit simplex. Moreover, $d!\cdot \vol_M(B)=\mult_x(X)$.

We claim that $S^{(p)}\subseteq M\cap pP\subseteq S^{(p-d+1)}$. Indeed, the first inclusion is clear.
For the second, let $m\in M\cap pP$. Then $m=\sum_{i\in I}t_im^i+m'$, where $t_i\ge 0$, $m^i \in S$,
$\sum_it_i=p$ and $m'\in \sigma^\vee$, and the cardinality of $I$ is at most $d$ 
(use Carath\'eodory's theorem). Then $m=\sum_i\lfloor t_i\rfloor m^i+(m'+\sum_i\{t_i\}m^i)$,
and $\sum_i \lfloor t_i\rfloor>\sum_i (t_i-1)=p-d$. Therefore $m\in S^{(p-d+1)}$.
 
We obtain 
$
\Gamma(\cI_x^p(qL))\subseteq \oplus_{m\in M\cap q\square_L\cap pP+qu}k\cdot\chi^m
\subseteq \Gamma(\cI_x^{p-d+1}(qL)).
$
Therefore the two algebras 
$$
\oplus_{p>qt}\Gamma(\cI_x^p(qL))\subseteq \oplus_{p> qt} \oplus_{m\in M\cap q\square_L\cap pP+qu}k\cdot\chi^m
$$
have the same stable base locus. Denote $t^+P=\cap_{\epsilon>0}(t+\epsilon)P$. We deduce that
near $x$, 
$\Bs|\cI_x^{t+}L|_\Q$ is the intersection of $\Supp(\chi^m)$, where 
$m\in (\square_L-u)\cap t^+P \cap M_{\Q}$ and $(\chi^m)$ is the effective $\Q$-divisor defined by $\chi^m$. 

\begin{lem}\label{des}
Near $x$, $\Bs|\cI_x^{t+}L|_\Q$ is the union of invariant closed irreducible subvarieties $x\in Y\subseteq X$
such that $\width_x(R(L)|_Y)\le t$. If $L$ is ample, the latter inequality means $\width_x(L|_Y)\le t$.
\end{lem}

\begin{proof}
$\Bs|\cI_x^{t+}L|_\Q$ is the intersection of all $\Supp(\chi^m)$,
where $\chi^m \in \Gamma(\cO_X(qL)) $ with $\ord_x (\chi^m) > qt$.
In particular, $\Bs|\cI_x^{t+}L|_\Q$ is torus invariant.

Let $x \in Y \subseteq X$ be a torus invariant subvariety.
By definition,
$\width_x(R(L)|_Y) > t$ if and only if the subspace
$$
\{s|_Y ; s \in \Gamma(X,\cO_X (qL)), s|_Y \not =0, \ord_x(s|_Y) > qt \}\subseteq  \Gamma(Y,\cO_Y(qL|_Y))
$$
contains a non-zero element for some $q$.
This subspace is torus invariant, and hence has a basis consists of monomials.
Thus this subspace has a non-zero element if and only if there exists $\chi^m \in \Gamma(\cO_X (qL))$ with $\chi^m |_{Y}  \not = 0$ and $ \ord_x(\chi^m|_Y) > qt$.
For monomials, the restriction to $Y$ does not change the order,
that is, $ \ord_x(\chi^m|_Y) = \ord_x(\chi^m)$ if  $\chi^m |_{Y}  \not = 0$.
Hence $\width_x(R(L)|_Y) > t$ if and only if there exists $\chi^m \in \Gamma(\cO_X (qL))$ with $\chi^m|_Y \not = 0$ and $ \ord_x(\chi^m) > qt$ for some $q$,
which is equivalent to say that $Y$ is not contained in $\Bs|\cI_x^{t+}L|_\Q$.

 If $L$ is ample, the restriction map $\Gamma(qL) \to \Gamma(qL|_Y)$ is surjective for any $q$, 
and therefore $\width_x(R(L)|_Y)=\width_x(L|_Y)$.
\end{proof}

\begin{lem}
Let $x\in Y\subseteq X$ be an invariant closed subvariety corresponding to a face $\tau \prec \sigma$.
Then $\width_x(R(L)|_Y) = \min\{t\ge 0; (\square_L-u) \cap  \tau^{\perp}  \subseteq tB \cap  \tau^{\perp}  \}\in \Q$.
In particular,
$\width_x(L)=\min\{t\ge 0; \square_L-u\subseteq tB\}\in \Q$.
\end{lem}

\begin{proof}
By Lemma \ref{des}, 
$\width_x(R(L)|_Y)  \le t$ if and only if $Y \subseteq \Bs|\cI_x^{t+}L|_\Q$, if and only if $(\square_L-u)\cap t^+P \cap  \tau^{\perp} \cap M_{\Q}=\emptyset$.
Since $\square_L, u, P $ are rational, this condition is equivalent to $(\square_L-u)\cap t^+P \cap  \tau^{\perp} =\emptyset$, which is 
equivalent to $\square_L-u \cap \tau^{\perp} \subseteq tB \cap  \tau^{\perp}$.
\end{proof}

We obtain the following proposition which generalizes \cite[Corollary 4.2.2]{BDH+}.

\begin{prop} $\epsilon_i(L,x)$ is the minimum of $\width_x(R(L)|_Y)$, 
after all closed irreducible invariant subvarieties $x\in Y\subseteq X$ of codimension $i-1$.
In particular, $\epsilon_i(L,x)\in \Q$.
\end{prop}

Suppose $L$ is ample. Then $\width_x(R(L)|_Y)=\width_x(L|_Y)$, and therefore
$\epsilon_i(L,x)$ is the minimum of $\width_x(L|_Y)$, after all closed irreducible invariant subvarieties 
$x\in Y\subseteq X$ of codimension $i-1$. In particular, 
$
\epsilon_d(L,x)=\min\{(L\cdot C); x\in C \subseteq X \text{ invariant curve}  \}.
$
Moreover, we have inclusions 
$$
\epsilon_d B\subseteq \square_L-u\subseteq \epsilon_1 B,
$$ 
and $\epsilon_d,\epsilon_1$ are maximal and minimal, respectively, with this property.


\subsection{At a general point}


Let $\square\subset M_\R$ be the moment polytope of $L$. Then $\square-\square\subset M_\R$
is a $0$-symmetric compact convex set, of dimension $\kappa=\kappa(L)$. The Minkowski
successive minima of $(M,\square-\square)$ are
$$
\lambda_i=\lambda_i(M,\square-\square)=\sup\{t\ge 0; \dim M\cap t(\square-\square)<i\}.
$$
We have $0\le \lambda_1\le \lambda_2\le \cdots\le \lambda_\kappa<\lambda_{\kappa+1}=+\infty$.
The polar convex set $(\square-\square)^*\subset N_\R$ coincides with the set of linear functionals
$\varphi\in N_\R$ such that $\length \varphi(\square)\le 1$. It is unbounded if and only if $\kappa<\rank M$.
The Minkowski successive minima of $(N,(\square-\square)^*)$ are
$$
\lambda_i^*=\lambda_i(N,(\square-\square)^*)=\sup\{t\ge 0; \dim N \cap t\cdot (\square-\square)^*<i\}.
$$

\begin{exmp}
Consider $(X,L)=(\bP^d,\cO(w))$, where $w$ is a positive integer. Since the ambient
is homogeneous, $\epsilon_i(L,x)=\epsilon_i(L)$ for all $x\in X$ and $i$. Then 
$$
\epsilon_1(L)=\cdots=\epsilon_d(L)=\sqrt[d]{\vol(L)}=w.
$$
Let $e_1,\ldots,e_d$ be the standard basis of $\Z^d$. The moment polytope of $L$ is 
$$
\square=\{\sum_{i=1}^dx_ie_i; x_i\ge 0, \sum_{i=1}^d x_i\le w\}.
$$
The difference $\square-\square$ is the convex hull of $\pm w e_i (1\le i\le d), \pm w (e_i-e_j)\ (1\le i<j\le d)$,
which is contained in $\{\sum_{i=1}^dx_ie_i; |x_i|\le w, |x_i+x_j| \le w\}$. We compute 
$$
\lambda_i(\Z^d,\square-\square)=\frac{1}{w}\ (1\le i\le d).
$$
Let $e_1^*,\ldots,e_d^*$ be the dual basis of $\check{\Z}^d$. The polar body $(\square-\square)^*$ is 
$\{\sum_{i=1}^dx_i^*e_i^*; |x^*_i|\le 1/w, |x^*_i-x^*_j| \le 1/w\}$. We compute 
$$
\lambda_i(\check{\Z}^d,(\square-\square)^*)=w \ (1\le i\le d).
$$
\end{exmp}

\begin{exmp}\label{product}
Consider $(X,L)=((\bP^1)^d,\cO(w_1,\ldots,w_d))$, where $w_1\ge w_2\ge \cdots\ge w_d$ are positive integers.
Since the ambient is homogeneous, $\epsilon_i(L,x)=\epsilon_i(L)$ for all $x$ and $i$.
Let $x=[1:0]^d \in X$. Then $\Bs|\cI_x^{t+}L|_\Q$ consists of the invariant cycles
$Y$ through $x$ (affine spaces with coordinates $z_i\ (i\in I)$) such that $\sum_{i\in I}w_i\le t$. Therefore 
$
\epsilon_i(L)=\min_{|I|=d-i+1}\sum_{j\in I}w_j.
$
Since we ordered the weights, we obtain 
$$
\epsilon_i(L)=\sum_{j=i}^d w_j.
$$
The volume is 
$$
\vol(L)= d! \prod_{i=1}^dw_i .
$$
The moment polytope of $L$ is $\square=\prod_{i=1}^d[0,w_i]\subset \R^d$. Then 
$\square-\square=\prod_{i=1}^d[-w_i,w_i]$, so 
$$
\lambda_i(\Z^d, \square-\square)=\frac{1}{w_i}.
$$
The polar body $(\square-\square)^*$ is $\{x\in \check{\R}^d; \sum_{i=1}^d|x_i|w_i\le 1\}$. Therefore 
$$
\lambda_i(\check{\Z}^d, (\square-\square)^*)=w_{d-i+1}.
$$
We obtain 
$$
\frac{\epsilon_i(L)}{\lambda_{d-i+1}^*}=\epsilon_i(L) \cdot \lambda_i=\frac{w_i+\cdots+w_d}{w_i}\in [1,d-i+1].
$$
\end{exmp}

\begin{lem}\label{Ss}
$\epsilon_j(L)\cdot \lambda_j\ge 1$.
\end{lem}

\begin{proof}
Denote $\mu=1/\lambda_j(M,\square-\square)$. There exist $u_1,\ldots,u_j \in M$, primitive and
linearly independent, such that $\mu u_i=m'_i-m_i$ for some $m_i,m'_i \in \square$.
The inclusion $[m_i,m'_i] \subset \square$ induces a dominant rational map 
$$
\varphi_i \colon (X,L)\dashrightarrow (\bP^1,\cO(\mu)), \overline{L}\ge \overline{\cO_{\bP^1}(\mu)}.
$$
Therefore there exists $\mu F_i+D_i\in |L|_\Q$, where $F_i$ is the fiber of $\varphi_i$ through $1$,
and $D_i$ is an effective invariant divisor on $X$. We have $F_i\cap T_N=T_{N\cap u_i^\perp}$.
Since $u_i$ are linearly independent, we have $\codim_1 \cap_{i=1}^j F_i=j$. 

We obtain $\codim_1 \Bs|\cI_1^\mu L|_\Q \ge j$. Therefore $\mu\le \epsilon_j(L)$.
\end{proof}

\begin{lem}\label{Ss'}
$\epsilon_{d-j+1}(L)\le j\cdot \lambda_j^*$.
\end{lem}

\begin{proof}
Suppose $\lambda_j^*=\lambda>0$. There exist
$\varphi_1,\ldots,\varphi_j\in N\cap \lambda\cdot (\square-\square)^*$, linearly independent.
The property $\varphi_i\in \lambda\cdot (\square-\square)^*$ means that the interval $\varphi_i(\square)$
has length at most $\lambda$. Consider the induced homomorphism of lattices
$$
(\varphi_1,\ldots,\varphi_j)\colon M\to \Z^j,
$$
whose image has rank $j$ since $\varphi_i$ are linearly independent.
Then $\square$ is mapped onto a polytope of dimension $j$, contained in 
$\prod_{i=1}^j[x_i,x_i+\lambda]$ for some $x_1,\ldots,x_j\in \R$. 
The lattice homomorphism induces a dominant rational map $X\dashrightarrow Y^{d-j}$ whose 
fiber $(F,L|_F)$ through $x=1$ has dimension $j$, and is dominated by $(\bP^1,\cO_{\bP^1}(\lambda))^j$. 
Therefore 
$$
\epsilon_1(L|_F,1)\le j\lambda
$$
by Example~\ref{product}.
Let $D\in |\cI^{j\lambda+}_1L|_\Q$. Then $D|_F\in |\cI^{j\lambda+}_1(L|_F)|_\Q$, which is empty since 
$\epsilon_1(L|_F,1)\le j\lambda$. Therefore $D|_F=0$. We deduce 
$$
F\subseteq \Bs |\cI^{j\lambda+}_1L|_\Q.
$$ 
Therefore $\codim_1 \Bs |\cI^{j\lambda+}_1 L|_\Q\le d-j<d-j+1$. Then $\epsilon_{d-j+1}(L)\le j\lambda$.
\end{proof}

\begin{rem}
For $\epsilon_d(L)$,
Lemmas~\ref{Ss}, \ref{Ss'} state that $\lambda_d^{-1} \leq  \epsilon_d(L) \leq \lambda^*_1$.
These inequalities also follow from \cite[Theorem 3.6]{AI}.
\end{rem}

\begin{thm} The invariants $\epsilon_i(L),1/\lambda_i,\lambda^*_{d-i+1}$ are all equivalent. More precisely,
$$
1\le \epsilon_i(L)\cdot \lambda_i \le d\cdot \frac{\epsilon_i(L)}{\lambda_{d-i+1}^*}\le d(d-i+1).
$$
\end{thm}

\begin{proof}
Use Banaszczyk's bound~\cite{Ban93} in Mahler's  transference theorem $1\le \lambda_i\lambda_{d-i+1}^*\le d$ for the 
second inequality, and the two lemmas above.
\end{proof}

\begin{thm}\label{ew}
The Seshadri constant of $L$ at a very general point 
is proportional to the lattice width of the moment polytope $\square_L$. More precisely,
$$
\frac{\width(\square_L)}{d} \le \epsilon(L) \le \width(\square_L).
$$
\end{thm}

\begin{proof} If $L$ is not big, both invariants are zero. Suppose $L$ is big.
Then $\epsilon(L)=\epsilon_d(L)$ and $\width(\square_L)=\lambda_1^*$, and we can apply the above theorem.
\end{proof}

\begin{rem} For toric varieties, we may establish the inequalities 
$$
1 \le \frac{\vol(L)}{\prod_{i=1}^d\epsilon_i(L)}\le d!
$$
of Theorem~\ref{m2} as follows: the left hand side inequality is easy to see.
For the right hand side, recall the second main theorem of Minkowski, 
in the stronger form due to Davenport-Estermann:
$$
\frac{1}{d!}\le \vol_M(\square)\cdot \prod_{i=1}^d\lambda_i(M,\square-\square)\le 1.
$$
We have $\vol(L)=d!\vol_M(\square)$. From $\epsilon_i \lambda_i\ge 1$, we obtain 
$\prod_i \epsilon_i\cdot \prod_i \lambda_i\ge 1$. Therefore $\vol(L) \le d! \prod_i \epsilon_i$.
\end{rem}


\section{Adjoint linear systems and the Flatness Theorem of Khinchin}


Recall first results of Demailly~\cite{De92} and Ein, K\"uchle, Lazarsfeld~\cite{EKL}:

\begin{thm}\label{aj} 
Let $X$ be a smooth projective variety of dimension $d$, let $L$ be a nef and big 
$\Q$-Cartier divisor on $X$ whose fractional part has normal crossing support. Let $\epsilon$ be the Seshadri constant of $L$ at a very general
point $x\in X$, which coincides with $\epsilon_d(L,x)$. The following properties hold:
\begin{itemize}
\item[1)] The jet map $\Gamma(\lceil K_X+L \rceil)\to \cO_x/\cI_x^{1+\lceil \epsilon-d-1\rceil}$
is surjective. In particular, 
$$
\dim_k \Gamma(\lceil K_X+L \rceil)\ge \binom{\lceil \epsilon-1\rceil}{d}=\frac{1}{d!}\prod_{i=1}^d\lceil \epsilon-i \rceil.
$$
\item[2)] If $\epsilon>d$, then $\Gamma(\lceil K_X+L \rceil)\ne 0$.
\item[3)] If $\epsilon>d+1$, then $|\lceil K_X+L \rceil|$ maps $X$ onto a variety of dimension $d$.
\item[4)] If $\epsilon>2d$, then $|\lceil K_X+L \rceil|$ maps $X$ birationally onto a variety of dimension $d$.
\end{itemize}
\end{thm}

\begin{proof} 1) Let $f\colon Y\to X$ be the blow-up at $x$, with exceptional divisor $E$.
Since $L$ has integer coefficients near $x$, $\lceil K_Y+f^*L\rceil=f^*\lceil K_X+L\rceil+(d-1)E$.
We may suppose $p=\lceil \epsilon-d-1\rceil$ is non-negative. Then
$
\lceil K_Y+f^*L-(p+d)E\rceil=f^*\lceil K_X+L\rceil-(p+1)E
$
and
$$
f_*\cO_Y(\lceil K_Y+f^*L-(p+d)E\rceil)=\cI_x^{p+1}(\lceil K_X+L\rceil).
$$
By the Leray spectral sequence, the natural homomorphism
$$
H^1(X,\cI_x^{p+1}(\lceil K_X+L\rceil))\to H^1(Y,\lceil K_Y+f^*L-(p+d)E\rceil)
$$
is injective. 

But $f^*L-(p+d)E=f^*L-(\lceil \epsilon \rceil-1)E$ is nef and big, since $\lceil \epsilon \rceil-1<\epsilon$. 
By Kawamata-Viehweg vanishing, $H^1(Y,\lceil K_Y+f^*L-(p+d)E\rceil)=0$. Therefore $H^1(X,\cI_x^{p+1}(\lceil K_X+L\rceil))=0$.
Then $\Gamma(\lceil K_X+L\rceil )\to \cO_x/\cI_x^{p+1}$ is surjective. Since $x$ is a smooth point, the 
right hand side has dimension $\binom{p+d}{d}$.

2) This follows from 1).

3) The $1$-jet map $\Gamma(\lceil K_X+L\rceil )\to \cO_x/\cI_x^2$ is surjective. 
Therefore $|\lceil K_X+L\rceil |$ moves, and the induced rational map is generically finite.

4) Let $x,y$ be two very general points, let $f\colon Y\to X$ be the blow-up at $x,y$. Then 
$$
f^*L-dE_x-dE_y=\frac{1}{2}(f^*L-2dE_x)+\frac{1}{2}(f^*L-2dE_y)
$$
is nef and big. By Kawamata-Viehweg vanishing, $H^1(Y,\lceil K_Y+f^*L-dE_x-dE_y\rceil)=0$. 
But $\lceil K_Y+f^*L-dE_x-dE_y\rceil=f^*\lceil K_X+L\rceil-E_x-E_y$. Therefore $H^1(X,\cI_x\otimes \cI_y(\lceil K_X+L\rceil))=0$.
\end{proof}

We generalize the flatness theorem of Khinchin. The original statement
says that a convex body which contains no lattice points must have lattice width bounded
above by a constant which depends only on the dimension (see~\cite{KL88} and part a) of the theorem below). 
We show that the same conclusion holds if the lattice points of the convex body are degenerate (part b) 
of the theorem below).

\begin{thm}\label{ft}
Let $M\simeq \Z^d$ be a lattice, let $\square\subset M_\R$ be a compact convex set, of dimension $d$.
Let $w$ be the lattice width of $\square$ with respect to $M$.
\begin{itemize}
\item[a)] If $w>d^2$, then $M\cap \Int\square\ne \emptyset$.
\item[b)] If $w>d(d+1)$, then $\dim(M\cap \Int\square)=d$.
\item[c)] If $w>2d^2$, then $M\cap \Int \square$ spans $M$.
\item[d)] $\sqrt[d]{d! |M\cap \Int \square|}\ge \frac{w}{d}-d$.
\item[e)] $\sqrt[d]{d! \vol_M(\square)}\ge \frac{w}{d}$.
\end{itemize}
\end{thm}

\begin{proof}
The width is continuous with respect to the approximation $\lim_{n\to \infty} \Conv(\frac{1}{n}M\cap \square)= \square$.
Therefore we may suppose $\square$ is the convex hull of finitely many points in $M_\Q$. 

Let $X$ be a toric
desingularization of the projective model of the graded ring $\oplus_{n\ge 0}\oplus_{m\in M\cap n\square}k \cdot \chi^m$.
There exists a $\Q$-Cartier divisor $L$ on $X$, semiample and big, supported by the invariant prime divisors of $X$,
such that $\square$ is the moment polytope of $L$. We compute 
$$
\Gamma(\lceil K_X+L\rceil)=\oplus_{m\in M\cap \Int\square} k\cdot \chi^m.
$$
Therefore the semi-invariant basis of $\Gamma(\lceil K_X+L\rceil)$ is in one-to-one correspondence to the interior
lattice points of $\square$. Denote $A=M\cap \Int\square$. The linear system $|\lceil K_X+L \rceil |$ is non-empty if
and only if $A$ is non-empty, it maps $X$ onto a variety of the same dimension if and only if $A-a$ generates the $\R$-vector space 
$M_\R$, for every $a\in A$, and maps $X$ birationally onto a variety of the same dimension if and only if $A-a$ generates the lattice 
$M$, for every $a\in A$.

Let $\epsilon$ be the Seshadri constant of $L$ at a very general point of $X$. We have
$\epsilon=\epsilon_d(L)$. By Theorem~\ref{ew}, $\epsilon\ge \frac{w}{d}$.
Therefore the claims are just a restatement of Theorem~\ref{aj} for $(X,L)$.
\end{proof}

Part a) of the theorem was proved by Kannan, Lov\'asz with $d^2$ replaced by $c d^2$, for some constant $c$.
It is expected that the optimal bound is linear in $d$ (see~\cite{KL88}). Parts b) and c) are new statements, 
while d) and e) improve similar bounds in~\cite{KL88}.


\section{Appendix on Geometry of Numbers}


We recall the definitions and results from the Geometry of Numbers that we use.
For proofs and more, the reader may consult~\cite{Lek69}.
Let $M\simeq \Z^d$ be a lattice of rank $d$. The dual lattice $N=\check{M}$ is defined
as $\Hom_\Z(M,\Z)$, and we have a duality pairing 
$
N \times M\to \Z,\ \langle \varphi,m\rangle= \varphi(m).
$

A subset $A\subseteq M$ spans the lattice $M$ if the difference set $A-A=\{a'-a;a',a\in A\}$ generates the lattice $M$.
The dimension of a subset $A\subseteq M_\R$, denoted $\dim(A)$, is the dimension of the $\R$-vector space generated by $A-A$. The $\Q$-dimension of a subset $A\subseteq M_\R$, denoted 
$\dim_\Q(A)$, is the dimension of the $\Q$-vector space generated by $A\cap M_\Q-A\cap M_\Q$. 
We have $\dim_\Q(A)\le \dim(A)$, and equality holds if $A\subseteq M_\Q$. 
If $\square\subseteq M_\R$ is a convex set, then $\dim_\Q(\square)=d$ if and only if $\dim(\square)=d$.

For a convex set $\square\subseteq M_\R$, the {\em polar convex set} $\square^*\subseteq N_\R$ is defined as
$$
\square^*=\{\varphi\in N^*_\R; \langle \varphi,m\rangle\ge -1\ \forall m\in \square  \}.
$$
To $\square$ we can associate the difference convex set $\square-\square=\{m'-m;m',m\in \square\}$, which is $0$-symmetric.
The polar convex set $(\square-\square)^*$ consists of the linear functionals
$\varphi\in N_\R$ such that the interval $\varphi(\square)$ has length at most $1$.

The {\em lattice width of a compact convex set} $\square\subset M_\R$ is defined as the smallest
length of any interval $\varphi(\square)$, with all $\varphi\in N\setminus 0$. It coincides with the 
first minimum of Minkowski of $(\square-\square)^*$ with respect to $N$.

Let $\square\subseteq M_\R$ be a closed convex set which contains the origin. For $i\ge 1$, 
the {$i$-th successive minimum of $(M,\square)$} is defined by
$$
\lambda_i(M,\square)=\sup\{t\ge 0;\dim (M\cap t\square)<i \}.
$$
We obtain an increasing chain
$
0\le \lambda_1\le \lambda_2\le \cdots \le \lambda_{d+1}=+\infty.
$
Note that $\lambda_i$ is a finite real number if and only if $i\le \dim_\Q(\square)$, and 
$\lambda_i(M,c\square)=\lambda_i(M,\square)/c$ for every $c>0$. In case $\dim(\square)=d$, we 
have the equivalent (Minkowski's) definition
$$
\lambda_i(M,\square)=\inf\{t\ge 0;\dim (M\cap t\square)\ge i \}.
$$

The second main theorem of Minkowski, in the generalized form due to Davenport-Estermann, 
is

\begin{thm}
Let $\square \subset M_\R$ be a convex set of dimension $d$. Then 
$$
\frac{1}{d!}\le \vol_M(\square)\cdot \prod_{i=1}^d\lambda_i(M,\square-\square)\le 1.
$$
\end{thm}

The transference theorem of Mahler, with the improved upper bound due to Banaszczyck~\cite[Theorem 2.1]{Ban93}, is 

\begin{thm}
Let $\square \subset M_\R$ be a $0$-symmetric compact convex set of dimension $d$. 
Let $\square ^*\subset N_\R$ be the polar body. Then 
$$
1\le \lambda_i(M,\square)\cdot \lambda_{d-i+1}(N,\square^*)\le d\ (1\le i\le d).
$$ 
\end{thm}

With the original upper bound of Mahler ($(d!)^2$ instead of $d$), the transference theorem
can be deduced from the second main theorem of Minkowski. The upper bound $d$ is sharp 
in dimension one. In dimension two, Banaszczyck~\cite[Proof of Theorem 2.1]{Ban93} mentions 
the upper bound $2/\sqrt{3}$. We show next that the sharp upper bound in dimension two is in fact $3/2$.


\subsection{Sharp transference theorem in dimension two}


Let $\square \subset \R^2$ be a $0$-symmetric polytope, of dimension $2$, with $\lambda_1(\Z^2,\square)=1$. 
This means that $\Z^2\cap \Int \square=\{0\}$ and $\Z^2\cap \partial \square \ne \emptyset$.

We say that $\square$ is {\em maximal} if every top face of $\square$ contains a lattice point in its relative interior.
We bring $\square$ into maximal position, preserving the initial hypothesis, as follows. Since $\square$
is $0$-symmetric, the top faces come in pairs. Fix a pair $F,-F$. We slide out these top faces until 
they either a) hit a lattice point in their relative interior, or b) these faces collapse to a point (they become 
redundant). Apriori, it is also possible that we may slide out the faces to infinity, but one may easily
check that this is impossible. We repeat the argument for all pairs of top faces.

In finitely many steps, we enlarged $\square\subset \square'$ such that $\square'\subset \R^2$ is a $0$-symmetric polytope, 
of dimension $2$, with $\lambda_1(\Z^2,\square')=1$, and $\square'$ is maximal. It is enough to prove transference for 
$\square'$.

From now on, we suppose $\square$ is maximal.

\begin{lem} On each top face $F$ of $\square$ choose exactly one relative interior lattice point $u_F$, in such a way so that
$-u_F$ is the choice for the opposite face $-F$. Let $Q$ be the convex hull of these lattice points. Then, after possibly 
changing the basis of $\Z^2$, $Q$ is one of the following:
\begin{itemize}
\item[1)] The square with vertices $(1,0),(0,1),(-1,0),(0,-1)$.
\item[2)] The polytope with vertices $(1,0),(0,1),(-1,1),(-1,0),(0,-1),(1,-1)$.
\end{itemize}
\end{lem}

\begin{proof}
Let $u_1,u_2$ be two adjacent vertices of $Q$. Then $\Conv(0,u_1,u_2)\setminus \{u_1,u_2\}$ is contained in 
$\Int \square$. Therefore $\Conv(0,u_1,u_2)$ contains no other lattice points besides its vertices. Therefore $u_1,u_2$
is a basis of $\Z^2$. We may suppose $u_1=(1,0)$ and $u_2=(0,1)$. Thus $u_1u_2$ is a top face of $Q$.
The opposite is also a top face. The other top faces are contained in the second and fourth quadrant. Enough to
look inside the second quadrant.

If $Q$ has $4$ vertices, we obtain 1). Suppose $Q$ has at least $6$ vertices. Then there are at least two top faces
in the second quadrant. But $(0,1)$ and $(-1,0)$ are vertices of $Q$, and $(-1,1)$ is not contained in the interior of $Q$.
Therefore $Q$ has exactly one vertex, namely $(-1,1)$, in the interior of the second quadrant. We are in case 2).
\end{proof}


\subsection{Case 1) }


 Here $\square=\cap_{i=1}^2 \{(x,y)\in \R^2; |\langle u_i,(x,y)\rangle|\le 1\}$, where $u_1=(1,\beta_1)$,
$u_2=(-\beta_2,1)$ and $\beta_1,\beta_2\in [0,1)$. Indeed, we may suppose $(1,1)$ is separated from $\square$ by the
edge passing through $(1,0)$. The exterior normal to this edge is $u_1=(1,\beta_1)$ for some $\beta_1\ge 0$.
Since $\langle \pm u_1,(0,1)\rangle<1$, we deduce $\beta_1<1$. The lattice point $(-1,1)$ must be separated from 
$\square$ by its edge through $(0,1)$. The exterior normal to this edge is $u_2=(-\beta_2,1)$ for some $\beta_2\ge 0$.
Since $\langle \pm u_1,(1,0)\rangle<1$, we deduce $\beta_1<1$. 
Conversely, any polytope $\square$ as above satisfies $\Z^2\cap \Int \square=\{0\}$. 

Since $u_1,u_2$ are linearly independent, we compute
$$
\square^*=\{\alpha_1 u_1+\alpha_2 u_2;|\alpha_1|+|\alpha_2|\le 1\}=\Conv(\pm u_1,\pm u_2).
$$
Let $z=(z_1,z_2)\in \check{\Z}^2$. Then 
$$
z=\frac{z_1+\beta_2z_2}{1+\beta_1\beta_2}u_1+\frac{z_2-\beta_1z_1}{1+\beta_1\beta_2}u_2
$$
and $z\in h(z)\cdot \square^*$, where 
$$
h(z)=\frac{|z_1+\beta_2z_2|}{1+\beta_1\beta_2}+\frac{|z_2-\beta_1z_1|}{1+\beta_1\beta_2}.
$$
We compute 
$$
h(1,0)=\frac{1+\beta_1}{1+\beta_1\beta_2}, \
h(0,1)=\frac{ 1+\beta_2 }{1+\beta_1\beta_2}, \
h(1,1)=\frac{ 2+\beta_2-\beta_1 }{1+\beta_1\beta_2}, \
h(-1,1)=\frac{ 2+\beta_1-\beta_2 }{1+\beta_1\beta_2}.
$$
We claim that $\lambda_2(\check{\Z}^2,\square^*)\le \frac{3}{2}$. Indeed, suppose $0\le \beta_1\le \beta_2<1$. Then 
$$
1+\beta_1\le 1+\beta_2,2+\beta_1-\beta_2 \ \text{ and } 1+\beta_2,2+\beta_1-\beta_2\le 2-\beta_1+\beta_2.
$$
Therefore
$$
\lambda_2(\check{\Z}^2,K^*)\le \min(\frac{ 1+\beta_2 }{1+\beta_1\beta_2}, \frac{ 2+\beta_1-\beta_2 }{1+\beta_1\beta_2}).
$$
One may check that the right hand side is at most $\frac{3}{2}$, with equality only for $\beta_1=0,\beta_2=\frac{1}{2}$.


\subsection{Case 2) }


Here $\square=\cap_{i=1}^3 \{(x,y)\in \R^2; |\langle u_i,(x,y)\rangle|\le 1\}$, where $u_1=(1,\beta_1)$,
$u_2=(\beta_2,1)$, $u_3=(-\beta_3,1-\beta_3)$, and $\beta_1,\beta_2,\beta_3\in (0,1)$. Indeed, we may 
suppose $(1,1)$ is separated from $\square$ by its edge through $(1,0)$. Let $u_1=(1,\beta_1)$ be the exterior normal,
so that $\beta_1\ge 0$. We have $\langle u_1,(0,1)\rangle<1$ and $\langle u_1,(1,-1)\rangle<1$. That is $0<\beta_1<1$. 
Let $u_2$ and $u_3$ be the normalized exterior normals to $Q$ through $(0,1)$ and $(-1,1)$, respectively. 
We have $\langle u_2,(1,0)\rangle<1$ and $\langle u_1,(-1,1)\rangle<1$, that is $0<\beta_2<1$. Similarly,
$\langle u_3,(0,1)\rangle<1$ and $\langle u_3,(-1,0)\rangle<1$, that is $0<\beta_3<1$.

Conversely, any polytope $\square$ as above satisfies $\Z^2\cap \Int \square=\{0\}$. Indeed, once $u_1$ is chosen,
it is enough to consider the lattice points in the strip $|\langle u_1,\cdot \rangle | \le 1$. By symmetry, we may 
take $y>0$. The lattice points in the strip with $x\le -1$ must be separated from $\square$ by its top face through $(-1,1)$.
We are left with the lattice points $(0,n)\ (n=1,2,3,\ldots)$. But these are separated from $\square$ by its edge through $(0,1)$.

Note that $\square^*$ is the $0$-symmetric polytope with vertices $\pm u_1,\pm u_2,\pm u_3$, and is inscribed in the polytope
$Q^*=\Conv(\pm (1,0),\pm (1,1),\pm (0,1),\pm (-1,1))$. For $e\in \check{\Z}^2\setminus 0$, let 
$h(e)=\inf\{t>0;e\in t\cdot \square^*\}$. Denote $s_e=(1-h(e)^{-1})^{-1}$. We compute
$$
s_{1,0}=\frac{1}{\beta_2}+\frac{1}{\beta_3}, \
s_{1,1}=\frac{1}{1-\beta_1}+\frac{1}{1-\beta_2}, \
s_{0,1}=\frac{1}{\beta_1}+\frac{1}{1-\beta_3}.
$$
For $t_1,t_2>0$, we have $\frac{1}{t_1}+\frac{1}{t_2}\ge \frac{4}{t_1+t_2}$. Thus
$$
s_{1,0}+s_{1,1}=\frac{1}{\beta_2}+\frac{1}{1-\beta_2}+ \frac{1}{1-\beta_1}+\frac{1}{\beta_3}\ge 4+\frac{1}{1-\beta_1}+\frac{1}{\beta_3}>6.
$$
Similarly, $s_{1,0}+s_{0,1}>6$ and $s_{1,1}+s_{0,1}>6$. Therefore two of the $s_{1,0},s_{1,1},s_{0,1}$ are strictly larger than $3$.
Therefore $\lambda_2(\check{\Z}^2,\square^*)< \frac{3}{2}$. We obtained:

\begin{thm}
Let $\square\subset \R^2$ be a $0$-symmetric convex body. Then 
$$
1\le \lambda_1(\Z^2,\square)\cdot \lambda_2(\check{\Z}^2,\square^*)\le \frac{3}{2}.
$$
\end{thm}

\begin{proof}
The first inequality is trivial. For the second, suppose by contradiction that it fails. Then we may approximate $\square$ with a 
polytope, and suppose $\square$ is a polytope. If $\square$ is a polytope, we may scale so that $\lambda_1=1$, and we are
done by the two cases above.
\end{proof}

\begin{exmp}
The upper bound $\frac{3}{2}$ is attained for $\square$ in Case 1) with $\beta_1=0,\beta_2=\frac{1}{2}$.
In this case, $\square = \Conv(\pm(1,\frac32), \pm(1,\frac12))$,  $\square^*=  \Conv(\pm(1,0), \pm(-\frac12,1))$ and 
$  \lambda_1(\Z^2,\square)=1, \lambda_2(\check{\Z}^2,\square^*) = \frac32$.
\end{exmp}


\end{document}